\documentclass{amsart}

\usepackage{hyperref}
\hypersetup{
    colorlinks=true,
    linkcolor=blue,
    citecolor=blue,
    urlcolor=blue,
}

\makeatletter
\newcommand\org@maketitle{}
\newcommand\@authors{}
\let\org@maketitle\maketitle
\def\maketitle{%
	\let\@authors\authors
	\nxandlist{; }{ and }{; }\@authors
	\hypersetup{
		linktocpage=true,
		pdftitle={\@title},
                pdfauthor={\@authors},
                pdfsubject={\subjclassname. \@subjclass},
		pdfkeywords={\@keywords}
	}%
	\org@maketitle
}
\makeatother
\usepackage{comment}
\usepackage{amssymb}
\usepackage{rsfso}
\usepackage{mathtools}
\usepackage{microtype}
\usepackage[alphabetic,msc-links]{amsrefs}
\usepackage{enumitem}
\usepackage{amsfonts}
\usepackage{color}

\usepackage{doi}
\renewcommand{\PrintDOI}[1]{\doi{#1}}

\newcommand{\arxiv}[1]{arXiv:\href{https://arxiv.org/abs/#1}{#1}}
\usepackage[
  hmarginratio={1:1},     
  vmarginratio={1:1},     
  textwidth=17cm,        
  textheight=21cm,
  heightrounded,          
]{geometry}

\numberwithin{equation}{section}

\newtheorem{theorem}{Theorem}[section]
\newtheorem{teo}[theorem]{Theorem}
\newtheorem{lem}[theorem]{Lemma}
\newtheorem{cor}[theorem]{Corollary}
\newtheorem{pro}[theorem]{Proposition}
\theoremstyle{definition}
\newtheorem{defi}[theorem]{Definition}

\theoremstyle{remark}
\newtheorem{oss}[theorem]{Remark}

\newcommand{\e}{\varepsilon}

\newcommand{\R}{\mathbb{R}}
\newcommand{\N}{\mathbb{N}}

\newcommand{\D}{\nabla}

\newcommand{\supp}{\operatorname{spt}}
\renewcommand{\div}{\operatorname{div}}
\newcommand{\loc}{\mathrm{loc}}

\DeclareMathOperator*{\esssup}{ess\,sup}

\def\XXint#1#2#3{{\setbox0=\hbox{$#1{#2#3}{\int}$}
    \vcenter{\hbox{$#2#3$}}\kern-.5\wd0}}

\mathchardef\ordinarycolon\mathcode`\:
\mathcode`\:=\string"8000
\begingroup \catcode`\:=\active
  \gdef:{\mathrel{\mathop\ordinarycolon}}
\endgroup

\author{Alessandro Audrito}
\author{Gabriele Fioravanti}
\author{Stefano Vita}

\address{Alessandro Audrito\newline\indent
Dipartimento di Scienze Matematiche ``Giuseppe Luigi Lagrange''
\newline\indent
Politecnico di Torino
\newline\indent
Corso Duca degli Abruzzi 24, 10129, Torino, Italy}
\email{alessandro.audrito@polito.it}
\address{Gabriele Fioravanti\newline\indent
Dipartimento di Matematica ``Giuseppe Peano''
\newline\indent
Universit\`a degli Studi di Torino
\newline\indent
Via Carlo Alberto 10, 10124, Torino, Italy}
\email{gabriele.fioravanti@unito.it}
\address{Stefano Vita\newline\indent
Dipartimento di Matematica ``Felice Casorati''
\newline\indent
Universit\`a di Pavia
\newline\indent
Via Ferrata 5, 27100, Pavia, Italy}
\email{stefano.vita@unipv.it}

\title[Higher order Schauder estimates for degenerate or singular parabolic equations]{Higher order Schauder estimates for degenerate
or singular parabolic equations}


\allowdisplaybreaks

\begin{document}

\begin{abstract}
In this paper, we complete the analysis initiated in \cite{AudFioVit24} establishing some higher order $C^{k+2,\alpha}$ Schauder estimates ($k \in \N$) for a a class of parabolic equations with weights that are degenerate/singular on a characteristic hyperplane. The $C^{2,\alpha}$-estimates are obtained through a blow-up argument and a Liouville theorem, while the higher order estimates are obtained by a fine iteration procedure. As a byproduct, we present two applications. First, we prove similar Schauder estimates when the degeneracy/singularity of the weight occurs on a regular hypersurface of cylindrical type. Second, we provide an alternative proof of the higher order boundary Harnack principles established in \cite{BanGar16,Kuk22}.
\end{abstract}

\maketitle

\section{Introduction}

In this paper we complete the study started in \cite{AudFioVit24}, establishing some higher order Schauder regularity estimates for solutions to a special class of parabolic equations having weights which degenerate or explode on a characteristic hyperplane $\Sigma$ as $\mathrm{dist}(\cdot,\Sigma)^a$, where $a>-1$ is a fixed parameter. More precisely, for every $k\in\mathbb N$, we prove local regularity estimates in $C^{k+2,\alpha}_p$ (parabolic H\"older) spaces ``up to'' $\Sigma$ for weak solutions to
\begin{equation}\label{eq:1}
\begin{cases}
    y^a \partial_tu - {\div }(y^{a} A\nabla u) = y^a f + {\div}(y^{a} F) \quad &\text{in } Q_1^+ \\
    \displaystyle{\lim_{y\to0^+}} y^a(A\nabla u+F)\cdot e_{N+1}=0              &\text{on }\partial^0Q_1^+.
\end{cases}
\end{equation}
Here $N\ge1$, $(z,t)=(x,y,t)\in\R^N\times\R\times\R$, $\Sigma=\{y=0\}$ and $\mathrm{dist}(P,\Sigma)^a = y^a$. Further, $Q_1^+ := B_1^+\times I_1$ is the unit upper-half cylinder and $\partial^0 Q_1^+ = Q_1 \cap\{y=0\}$, where $B_1^+ := B_1\cap\{y>0\}$ ($B_1\subset\R^{N+1}$ is the unit ball centered at $0$) and $I_1:=(-1,1)$, while the symbols $\nabla$ and $\div$ denote the gradient and the divergence w.r.t. the spatial variable $z$, respectively.

The function $A:Q_1^+\to \R^{N+1,N+1}$ is assumed to be symmetric and to satisfy the following ellipticity condition: there exist $0 < \lambda \le \Lambda < +\infty$ such that
\begin{equation}\label{eq:UnifEll}
\lambda|\xi|^2\le A(z,t)\xi\cdot\xi\le\Lambda|\xi|^2,
\end{equation}
for all $\xi\in\mathbb{R}^{N+1}$ and a.e. $(z,t)\in Q_1^+$, while $f:Q_1^+ \to \R$ and $F:Q_1^+ \to \R^{N+1}$ are given functions belonging to some suitable functional spaces. The notion of weak solution  is given in Definition \ref{def.solution}. 

\

Our theory fits into the context of the regularity theory for linear non-uniformly parabolic equations; in particular, second order linear parabolic equations where the lack of uniform parabolicity is entailed by a weight term. Among all the papers on this topic, we quote the pioneering works \cite{FabKenSer82,ChiSer85} where Harnack estimates and local H\"older continuity of solutions have been established when the weight $\omega$ either comes from quasiconformal mappings or belongs to the $A_2$-Muckenhoupt class, that is,
\begin{equation*}\label{muckA2}
        \sup_{B} \Big(\frac{1}{|B|}\int_{B} \omega \Big)\Big(\frac{1}{|B|}\int_{B} \omega^{-1}\Big)\le C,
\end{equation*}
where the supremum is taken over every ball $B\subset\mathbb{R}^{N+1}$. 

The weight term $|y|^a$ we are considering here is $A_2$-Muckenhoupt in the range $a\in(-1,1)$.
However, the peculiar geometry of the degeneracy/singularity set of our weight - the characteristic hyperplane $\Sigma$ - allows us to get more information compared to the general theory quoted above and to deal with the full range $a>-1$.

In the spirit of the elliptic framework, see \cite{SirTerVit21a,SirTerVit21b,TerTorVit22}, one can build a complete Schauder theory in $C^{k,\alpha}_p$ spaces for weak solutions to \eqref{eq:1}: this is the main issue of the present paper, together with its first part \cite{AudFioVit24}. Let us remark here that the regularity we obtain strongly relies on the \textsl{natural conormal boundary condition} 
$$\lim_{y\to0^+}y^a(A\nabla u+F)\cdot e_{N+1}=0$$
we impose on the characteristic hyperplane $\Sigma$: the reader should keep in mind that the function $y^{1-a}$ is a solution to the homogeneous equation $\mathrm{div}(y^{a}\nabla(y^{1-a}))=0$ when $a<1$ with homogeneous Dirichlet boundary condition on $\Sigma$ but, if $a \in (0,1)$, it is no more than $(1-a)$-H\"older continuous up to $\Sigma$.

We also mention \cite{JeoVit23,DonJeoVit23} where Schauder estimates in the elliptic framework are obtained when data are of Dini type, and \cite{DongDirichlet,DonPha23} where the authors established some regularity estimates of Sobolev type for a wide class of parabolic equations including \eqref{eq:1} (see also \cite{MetNegSpi23a,MetNegSpi23b,MetNegSpi24,NegSpi24}).

Moreover, the study of weighted problems like \eqref{eq:1} is strongly related to the theory of \textsl{edge operators} \cite{Maz91,MazVer14}, and \textsl{nonlocal operators}. The latter relies in the connection between a class of fractional heat operators like $(\partial_t - \Delta)^{\frac{1-a}{2}}$ - possibly with variable coefficients - and their extension theories \cite{NysSan16,StiTor17,banerjee}, which represent the parabolic counterpart of \cite{CafSil07}. Within this context, Schauder estimates for solutions to fractional parabolic equations involving $(\partial_t-\div_x(A(x)\nabla_x))^{\frac{1-a}{2}}$ have been established in \cite{BisSti21}. Respect to our notation, this corresponds to regularity estimates in the $(x,t)$-variables on $\Sigma$ and $a\in(-1,1)$ (see also \cite{BucKar17,CafSti16,DonKim13,Sil12}). Let us also mention that space analyticity (in the full $z$ variable) and smoothness in $(z,t)$ of solutions to equation \eqref{eq:1} were already available by \cite{BanGar23} when $a\in(-1,1)$ and coefficients are analytic and satisfy suitable extra assumptions.

It is worth mentioning that the study of such operators is central in numerous papers of the last years: we quote \cite{audrito,CAFFA-MELLET-SIRE} (reaction-diffusion equations), \cite{DGPT17,AthCafMil18,BanDanGarPet21} (obstacle problems), \cite{SirTerTor20,audritoterracini} (nodal set analysis), \cite{HYDER} (nonlocal harmonic maps flow) and the references therein.

\smallskip

According to \cite{TerTorVit22} (elliptic setting), the Schauder estimates for equations with degenerate weights have a remarkable application in the context of the \emph{boundary Harnack principles}. Such boundary Harnack principles allow to ``compare the regularity'' of two solutions $u,v$ of the same equation ($u>0$) which vanish on the same portion of a fixed boundary. In particular, in rough domains such as Lipschitz, NTA or H\"older domains, the ratio $w=v/u$ is bounded up to the boundary where $u$ and $v$ vanish (in the first two cases $w$ is even H\"older continuous). The literature on the topic is extensive: we refer to \cite{DeSSav20,DeSSav22} for a unified approach (equations in divergence and non-divergence form) and an interesting review of the topic. Then, when the boundary is $C^{k,\alpha}$, the \emph{higher order boundary Harnack principle} improves the regularity of the quotient $w$ up to $C^{k,\alpha}$, see \cite{DeSSav15} for the elliptic case and \cite{BanGar16,Kuk22} for its parabolic counterpart. We will see that our Schauder estimates for weighted equations provides an alternative proof of some of the results contained in the last two references.

Notably, the weighted elliptic Schauder theory developed in \cite{SirTerVit21a,TerTorVit22} was used in the recent papers \cite{AllKriSha24} and \cite{ResRos24} to derive higher regularity of free interfaces for some semilinear free boundary problems (Alt-Phillips type). We wonder if the parabolic Schauder theory we develop here, together with \cite{AudFioVit24}, may help to address similar results for semilinear free boundary problems of parabolic type as well.

\subsection*{Main results}
This paper is devoted to the higher order Schauder estimates for weak solutions to \eqref{eq:1}. Below the statement of our main result.
\begin{teo}\label{teo1} 
    Let $N\ge1$, $a>-1$, $r\in(0,1)$, $\alpha\in(0,1)$ and $k\in\N$. Let $A\in C^{k+1,\alpha}_p(Q_1^+)$ satisfying \eqref{eq:UnifEll}, $f\in C^{k,\alpha}_p(Q_1^+)$ and $F\in C^{k+1,\alpha}_p(Q_1^+)$ and let $u$ be a weak solution to \eqref{eq:1}. Then, there exists $C>0$ depending only on $N$, $a$, $\lambda$, $\Lambda$, $r$, $\alpha$ and $\|A\|_{C^{k+1,\alpha}_p(Q_1^+)}$ such that
    \begin{equation}\label{eq:ck:1}
    \|u\|_{C^{k+2,\alpha}_p(Q_{r}^+)}\le C\Big(
    \|u\|_{L^2(Q_1^+,y^a)}+
    \|f\|_{C^{k,\alpha}_p(Q_1)}+
    \|F\|_{C^{k+1,\alpha}_p(Q_1)}
    \Big).
    \end{equation}
\end{teo}

In our previous work \cite{AudFioVit24}, we established $C^{0,\alpha}_p$ and $C^{1,\alpha}_p$ estimates for solutions to \eqref{eq:1}, under suitable assumptions on coefficients and data, see \cite{AudFioVit24}*{Theorem 1.1}. These are obtained through a regularization of the weight and approximation, that is, by proving uniform-in-$\varepsilon$ regularity estimates for solutions $u_\varepsilon$ of the equation with the regularized weight $(\varepsilon^2+y^2)^{a/2}$ and then passing to the limit as $\varepsilon\to0^+$. The strategy to prove $C^{2,\alpha}_p$ (or higher order) estimates cannot rely on such $\varepsilon$-regularization scheme, since the $\varepsilon$-stability of the $C^{2,\alpha}_p$ estimate is false in general, even in the elliptic framework, see \cite{SirTerVit21a}*{Remark 5.4}.

\smallskip

Before sketching the main steps of the proof of Theorem \ref{teo1}, it is important to highlight the following facts, which substantially differ our strategy from the existing literature:
\begin{itemize}
\item In the \emph{weighted elliptic framework} (see \cite{SirTerVit21a}), as soon as the $C^{1,\alpha}$ regularity is available, one can iterate it on derivatives. This is obtained in two steps: first, one notice that, since the weighted elliptic operator commutes with all but one derivatives, $\partial_{x_i}u$ is also a solution for any $i=1,...,N$ (and so $\partial_{x_i}u$ gains regularity); then, 
the operator itself gives the regularity of the last derivative $\partial_yu$. Formally, this is because, in the special case $A = I$, one can re-write the equation as
\[
-\partial_{yy} u -\partial_y F\cdot e_{N+1} - \frac{a}{y} (\partial_y u + F\cdot e_{N+1}) = f + \div_x F + \Delta_x u,
\]
and thus, if $\Delta_x u$ is smooth, then $\partial_yu$ is smooth by ODE methods (of course, provided that the data are smooth as well).

\item In the \emph{non-weighted parabolic framework} (see \cite{Lie96}), the idea is roughly the same: if $\Delta_x u$ is smooth, then the equation 
\[
\partial_t u = f + \div F + \Delta_x u 
\]
yields smoothness of $\partial_t u$. 

\item In the present \emph{degenerate parabolic setting}, the ``degenerate'' variables are two, $y$ and $t$, and the above strategies do not apply. In particular, the induction argument requires, as starting point, the $C^{2,\alpha}_p$ regularity of weak solutions (see Proposition \ref{teo-C^2,alpha}).
\end{itemize}
Given the above remarks, our approach  relies on a priori estimates and a regularization procedure by convolution with standard mollifiers. More precisely:

\smallskip

For the $C^{2,\alpha}_p$ regularity:
\begin{itemize}
\item[(i)] We establish some \emph{a priori} $C^{2,\alpha}_p$ estimates in Proposition \ref{teo-C^2,alpha} using a blow-up argument combined with a Liouville theorem (see Theorem \ref{teo:polynomial:liouville} below), in the spirit of \cite{simon} (see also \cite{SirTerVit21a} in the weighted  elliptic setting).
\item[(ii)] We prove $C^{2,\alpha}_p$ regularity of weak solutions when the data are $C^\infty$ (see Lemma \ref{lemma:approx:c2}). In this step, the $C^{1,\alpha}_p$ regularity of weak solutions (see Theorem \ref{teo-C^1,alpha}) is crucial.
\item[(iii)] We use an approximation scheme to regularize \eqref{eq:1}, by convolution of the data with a family of standard mollifiers. Along the approximating sequence, the $C^{2,\alpha}_p$ regularity estimate extends to weak solutions with $f\in C^{0,\alpha}_p$ and $A,F\in C^{1,\alpha}_p$. In other words, we prove the \emph{a posteriori} regularity estimate in Theorem \ref{teo1} when $k=0$.
\end{itemize}

For the $C^{k+2,\alpha}$ regularity for every $k\geq1$:
\begin{itemize}

\item[(iv)] When the forcing term is zero, i.e. $f=0$, we iterate the regularity estimates previously obtained - i.e. the $C^{1,\alpha}_p$ and $C^{2,\alpha}_p$ regularity - on partial derivatives of solutions,  by using the same scheme as in the proof of Lemma \ref{lemma:approx:c2} and Theorem \ref{teo1} follows quite easily. 

\item[(v)] In the case of general forcing terms $f\in C^{k,\alpha}_p$ the argument of (iv) doesn't apply (at least for $k=1$), and hence we proceed as follows: we use the procedure described at points (i), (ii), (iii) at any order $k$. To be more precise, the $C^{k+2,\alpha}_p$ \emph{a priori} estimates are obtained inductively on $k$, starting from the $C^{2,\alpha}_p$ \emph{a priori} estimates proved at point (i). This part crucially uses a delicate analysis of a second order weighted-type derivative of solutions in $y$ (see Lemma \ref{approximation C^3}). The $C^{k+2,\alpha}_p$ regularity when the data are smooth (the analogous of point (ii)) is also proved by induction in Lemma \ref{lemma:infinity-regularity}. Finally, with the same regularization argument in (iii), we finally obtain Theorem \ref{teo1}.
\end{itemize}

As anticipated above, the proof of the \emph{a priori} $C^{2,\alpha}_p$ estimates strongly relies on the following Liouville theorem.
\begin{teo}\label{teo:polynomial:liouville}
   Let $a>-1$, $m\in\mathbb{N}$, $\gamma \in [0,m+1)$ and let $u$ be an entire solution to
   \begin{equation}\label{eq-liouville-polinomiale}
       \begin{cases}
       y^a\partial_t u-{\div}(y^a \nabla u)=0&\text{in }\mathbb{R}^{N+1}_+\times\mathbb{R},\\
       \displaystyle{\lim_{y\to0+}}y^a\partial_y u=0&\text{on }\partial\R_+^{N+1}\times\R.
   \end{cases}
   \end{equation}
Assume that
   \begin{equation}\label{growth-polinomiale}
       |u(z,t)|\le C(1+(|z|^2+|t|)^\gamma)^{1/2} \quad \text{ for a.e. } (z,t) \in \mathbb{R}^{N+1}_+\times\mathbb{R}.
   \end{equation}
Then $u$ is a polynomial with degree at most $m$ in $z$ and at most $\lfloor \frac{m}{2}\rfloor$ in $t$.
\end{teo}

As a consequence of our main theorem, we can treat more general equations with weights behaving as \emph{distance functions} to a $C^{k+2,\alpha}$ ($k\in\mathbb N$) hypersurface $\Gamma \subset \R^{N+1}$ (curved characteristic manifolds) that we introduce below. The case of weights behaving as \emph{distance functions} to a $C^{1,\alpha}$  hypersurface is treated in \cite[Corollary 1.3]{AudFioVit24}.

Such equations are set in cylindrical domains $\Omega^+\times(-1,1)$ of $\R^{N+2}$ which ``live'' on one side of $\Gamma\times(-1,1)$. Specifically, up to rotations and dilations, $0\in\Gamma$ and there exist a spacial direction $y$ and a function $\varphi\in C^{k+2,\alpha}(B_1\cap\{y=0\})$ with $\varphi(0)=0$ and $\nabla_x\varphi(0)=0$ such that
\begin{equation}\label{phi}
\Omega^+\cap B_1=\{y>\varphi(x)\}\cap B_1,\qquad \Gamma\cap B_1=\{y=\varphi(x)\}\cap B_1.
\end{equation}
Then, the family of weights $\delta = \delta(z)$ we consider behave as a distance function to $\Gamma$ in the sense that $\delta\in C^{k+2,\alpha}(\Omega^+\cap B_1)$, and
\begin{equation}\label{delta}
\begin{cases}
\delta>0 &\mathrm{in \ } \Omega^+\cap B_1\\
|\nabla\delta|\geq c_0>0 &\mathrm{in \ } \Omega^+\cap B_1\\
\delta=0 &\mathrm{on \ } \Gamma\cap B_1,
\end{cases}
\end{equation}
and we consider weighted equations of the form
\begin{equation}\label{eq:1:curve}
\begin{cases}
    \delta^a \partial_tu - {\div }(\delta^{a} A\nabla u) = \delta^a f + {\div}(\delta^{a} F) \quad &\text{in } (\Omega^+\cap B_1)\times(-1,1),\\
    \delta^a(A\nabla u+F)\cdot \nu=0              &\text{on }(\Gamma\cap B_1)\times(-1,1),
\end{cases}
\end{equation}
where $\nu$ is the unit outward normal vector to $\Omega^+$ on $\Gamma$. For a precise definition of solutions to \eqref{eq:1:curve} see \cite[Definition 7.2]{AudFioVit24}.

\begin{cor}\label{cor:C^1,alpha}
    Let $a>-1$, $k\in\mathbb N$, $\alpha\in(0,1)$ and $u$ be a weak solution to 
    \eqref{eq:1:curve}. Let $\varphi\in C^{k+2,\alpha}(B_1\cap\{y=0\})$ be the parametrization defined in \eqref{phi} and $\delta\in C^{k+2,\alpha}(\Omega^+\cap B_1)$ satisfying \eqref{delta}. 
    
    Let $A,F\in C^{k+1,\alpha}_p((\Omega^+\cap B_1)\times(-1,1))$, with $A$ satisfying \eqref{eq:UnifEll}, $f\in C^{k,\alpha}_p((\Omega^+\cap B_1)\times(-1,1))$. Then, there exists a constant $C>0$, depending on $N$, $a$, $\lambda$, $\Lambda$, $\alpha$, $c_0$, $\|A\|_{C^{k+1,\alpha}_p((\Omega^+\cap B_1)\times(-1,1))}$, $\|\varphi\|_{C^{k+2,\alpha}(B_1\cap\{y=0\})}$ and $\|\delta\|_{C^{k+2,\alpha}(\Omega^+\cap B_1)}$ such that
    \begin{equation*}
    \begin{aligned}
    \|u\|_{C^{k+2,\alpha}_p((\Omega^+\cap B_{1/2})\times(- 1/2,1/2))}\le C\Big(
    &\|u\|_{L^2((\Omega^+\cap B_1)\times(-1,1),\delta^a)} \\
    & \quad + \|f\|_{C^{k,\alpha}_p((\Omega^+\cap B_1)\times(-1,1))}+
    \|F\|_{C^{k+1,\alpha}_p((\Omega^+\cap B_1)\times(-1,1))}
    \Big).
    \end{aligned}
    \end{equation*}
\end{cor}

Finally, following the program of the elliptic setting (see \cite{TerTorVit22}), we provide an alternative proof of some \emph{parabolic higher order boundary Harnack principles} as in \cite{BanGar16,Kuk22}. Such kind of ``regularity comparison principle'' between two caloric functions $u,v$ (or solutions to more general parabolic equations), vanishing on the same fixed boundary, can be viewed as the Schauder regularity of their quotient $w=v/u$ which, in turns, satisfies a parabolic equation with degenerate weight $u^2$, see \eqref{eq:w:point}. After proper diffeomorphic transformations of the domain, the Schauder theory for the ratio $w$ follows as a byproduct of our main Theorem \ref{teo1}.

The ``regularity comparison principle'' is localized at boundary points which lie on the \textsl{lateral parabolic boundary} of a space-time domain. In other words, let us consider $u,v$ solutions of
\begin{equation}\label{BH}
\begin{cases}
\partial_tu-\mathrm{div}(A\nabla u)=g+Vu+b\cdot\nabla u &\mathrm{in} \ \Omega\cap Q_1\\
\partial_tv-\mathrm{div}(A\nabla v)=f+Vv+b\cdot\nabla v &\mathrm{in} \ \Omega\cap Q_1\\
u(z,t)\geq c_0 \, d_p((z,t),\partial\Omega\cap Q_1) &\mathrm{in} \ \Omega\cap Q_1\\
u=v=0 &\mathrm{on} \ \partial\Omega\cap Q_1,
\end{cases}
\end{equation}
where $A$, $V$, $b$, $g$ and $f$ are suitable data (see Theorem \ref{BHk+2alpha} below). Here, up to rotations, dilations and translations, $0$ belongs to the parabolic lateral boundary of $\Omega$; that is, there exists a parametrization $\varphi$ such that
\begin{equation}\label{lateral_boundary}
\Omega\cap Q_1=\{y>\varphi(x,t)\},\qquad \partial\Omega\cap Q_1=\{y=\varphi(x,t)\},
\end{equation}
with $\varphi(0)=0$ and $\nabla_x\varphi(0)=0$. Moreover, the parabolic distance to the boundary is defined as
\begin{equation*}
d_p((z,t),\partial\Omega\cap Q_1)=\inf_{(\zeta,\tau)\in\partial\Omega\cap Q_1}d_p((z,t),(\zeta,\tau)),
\end{equation*}
and the parabolic distance between points is defined in \eqref{eq:par:dist}.

We will present here the parabolic higher order boundary Harnack principle for equations in divergence form in $C^{k+2,\alpha}_p$-domains, $k\in\mathbb N$. However, let us stress the fact that the regularity assumptions we make on boundaries, coefficients and right hand sides, always allows to pass from non divergence to divergence form equations and viceversa, interchangeably. So, we are considering the same conditions set in \cite{BanGar16}, which are slightly more general compared to \cite{Kuk22}, where the assumptions on the drift terms are suboptimal. Actually, our approach allows us to treat equations with nontrivial forcing terms $g$ in the r.h.s. of the equation of $u$.

\begin{theorem}\label{BHk+2alpha}

Let $k\in\mathbb N$, $\alpha\in(0,1)$ and $u,v$ be solutions to \eqref{BH}. Let $\varphi\in C^{k+2,\alpha}_p(Q_1\cap\{y=0\})$ be the parametrization defined in \eqref{lateral_boundary}. Let $A,f,g\in C^{k+1,\alpha}_p(\Omega\cap Q_1)$, with $A$ satisfying \eqref{eq:UnifEll}, $V,b\in C^{k,\alpha}_p(\Omega\cap Q_1)$.

    Then, there exists a constant $C>0$, depending on $N$, $\lambda$, $\Lambda$, $c_0$, $\alpha$, $\|A\|_{C^{k+1,\alpha}_p(\Omega\cap Q_1)}$, $\|g\|_{C^{k+1,\alpha}_p(\Omega\cap Q_1)}$, $\|V\|_{C^{k,\alpha}_p(\Omega\cap Q_1)}$, $\|b\|_{C^{k,\alpha}_p(\Omega\cap Q_1)}$,  $\|\varphi\|_{C^{k+2,\alpha}_p(Q_1\cap\{y=0\})}$  and $\|u\|_{L^2(\Omega\cap Q_1)}$ such that
    \begin{equation*}
    \begin{aligned}
    \left\|\frac{v}{u}\right\|_{C^{k+2,\alpha}_p(\Omega\cap Q_{1/2})}\le C\Big(
    \|v\|_{L^2(\Omega\cap Q_1)} + \|f\|_{C^{k+1,\alpha}_p(\Omega\cap Q_{1})} \Big).
    \end{aligned}
    \end{equation*}

\end{theorem}

\section{Preliminaries}
In this section we introduce some preliminary notions from \cite{AudFioVit24} (parabolic H\"older spaces, weak solutions, and so on). Further, we prove some auxiliary/technical results we will repeatedly use throughout the paper. 

We begin with the definitions of the parabolic H\"older spaces, see \cite{Lie96}*{Chapter 4} and \cite{LSU}*{Chapter 1}.
\subsection{Parabolic H\"older spaces}\label{section:parabolic:holder;spaces}
Let $\Omega\subset\mathbb{R}^{N+1}\times\R$ be an open subset and $u:\Omega\to\mathbb{R}$. The parabolic distance $d_p:\Omega\times\Omega\to\mathbb{R}$ is defined by
\begin{equation}\label{eq:par:dist}
d_p((z,t),(\zeta,\tau)) := (|z-\zeta|^2+|t-\tau|)^{1/2}, 
\end{equation}
for all $(z,t),(\zeta,\tau)\in\Omega$, where $z,\zeta\in\mathbb{R}^{N+1}$, $t,\tau\in\mathbb{R}$.
For $\alpha\in(0,1]$, we define the seminorms
$$
[u]_{C^{0,\alpha}_p(\Omega)} := \sup_{\substack{(z,t),(\zeta,\tau) \in\Omega \\ (z,t)\not=(\zeta,\tau)}}\frac{|u(z,t)-u(\zeta,\tau)|}{(|z-\zeta|^2+|t-\tau|)^{\alpha/2}}, \qquad [u]_{C^{0,\alpha}_t(\Omega)} := \sup_{\substack{(z,t),(z,\tau) \in\Omega \\ t\not=\tau}}\frac{|u(z,t)-u(z,\tau)|}{|t-\tau|^\alpha},
$$
and the norm
\[
\|u\|_{C^{0,\alpha}_p(\Omega)} := \|u\|_{L^\infty(\Omega)} + [u]_{C^{0,\alpha}_p(\Omega)}.
\]
If $\beta\in\mathbb{N}^{N+1}$ is a multi-index and $k \geq 1$, we define the seminorms 
$$
[u]_{C_p^{k,\alpha}(\Omega)}:=\sum_{|\beta|+2j=k}[\partial_x^\beta\partial_t^j u]_{C^{0,\alpha}_p(\Omega)} + [u]_{C^{k-1,\frac{1+\alpha}{2}}_t(\Omega)}, \qquad [u]_{C^{k,\frac{1+\alpha}{2}}_t(\Omega)} := \sum_{|\beta|+2j=k}[\partial_x^\beta\partial_t^j u]_{C^{0,\frac{1+\alpha}{2}}_t(\Omega)},
$$
and the norm
$$
\|u\|_{{C_p^{k,\alpha}(\Omega)}}=\sum_{|\beta|+2j\le k}\sup_{\Omega}|\partial_x^\beta\partial_t^j u|+[u]_{C_p^{k,\alpha}(\Omega)}.
$$
We set
\[
C^{k,\alpha}_p(\Omega) := \{u:\Omega \to \R: \|u\|_{C^{k,\alpha}_p(\Omega)} < +\infty \}.
\]
Finally, we recall some interpolation inequalities in parabolic H\"older spaces.
\begin{lem}[\cite{Lie96}*{Proposition 4.2}]
Let $N \geq 1$, and $0 < \beta < \alpha \le 1$. Then, for every $\e>0$ there exists $C > 0$ depending on $N$ and $\e$ such that
\begin{align}\label{lemma:interpolation}
\begin{aligned}
&\|u\|_{C^{0,\beta}_p(\Omega)}\le C \|u\|_{L^\infty(\Omega)}+ \e \|u\|_{C^{0,\alpha}_p(\Omega)}, \\
&\|\nabla u\|_{L^\infty(\Omega)}\le C \|u\|_{L^\infty(\Omega)}+\e [u]_{C^{1,\alpha}_p(\Omega)}, \\
&\| D^2 u\|_{L^\infty(\Omega)} + \|\partial_t u\|_{L^\infty(\Omega)}\le C \|u\|_{L^\infty(\Omega)} + \e [u]_{C^{2,\alpha}_p(\Omega)}. 
\end{aligned}
\end{align}
\end{lem}
\subsection{Weak solutions, energy estimates and $C_p^{1,\alpha}$ regularity}
Let $r>0$. In what follows, $B_r\subset\mathbb{R}^{N+1}$ denotes the ball of radius $r$ centered at the origin, $I_r:=(-r^2,r^2)\subset\R$, $Q_r:=B_r\times I_r \subset \mathbb{R}^{N+2}$ is the parabolic cylinder of radius $r$ centered at the origin and $Q_r^+ : =Q_r\cap\{y>0\}$, while $\partial^0Q_r^+:=Q_r\cap\{y=0\}$ is the flat boundary of the half cylinder.

We first recall the definition of weak solutions to problem \eqref{eq:1}, see \cite{AudFioVit24}*{Definition 2.15}. The weighted energy spaces $L^2(Q_r^+,y^a )$, $L^2(Q_r^+,y^a )^{N+1}$, $H^1(B_r^+,y^a)$, $L^2(I_r;H^1(B_r^+,y^a)) $, $ L^\infty(I_r;L^2(B_r^+,y^a))$ appearing below are defined in \cite[Section 2.1]{AudFioVit24}. 
\begin{defi}\label{def.solution}
Let $a>-1$, $N\ge1$, $r>0$, $f \in L^2(Q_r^+,y^a )$, $F \in L^2(Q_r^+,y^a )^{N+1}$. We say that $u$ is a weak solution to \eqref{eq:1} if $u \in L^2(I_r;H^1(B_r^+,y^a)) \cap L^\infty(I_r;L^2(B_r^+,y^a))$ and satisfies
\[
-\int_{Q_r^+}y^a u\partial_t\phi dzdt + \int_{Q_r^+}y^a A\nabla u\cdot\nabla\phi dzdt = \int_{Q_r^+}y^a( f\phi -F\cdot\nabla\phi )dzdt,
\]
for every $\phi\in C_c^\infty(Q_r)$.
We say that $u$ is an entire solution to 
\[
\begin{cases}
    y^a \partial_tu - {\div }(y^{a} A\nabla u) = y^a f + {\div}(y^{a} F) \quad &\text{in } \R_+^{N+1}\times\R \\
    \displaystyle{\lim_{y\to0^+}} y^a(A\nabla u+F)\cdot e_{N+1}=0              &\text{on }\partial \R_+^{N+1}\times\R,
\end{cases}
\]
if, for every $r > 0$, $u$ is a weak solution to \eqref{eq:1}.
\end{defi}
Weak solutions satisfy the following local energy inequality. We state the version we obtained in \cite{AudFioVit24} in the spirit of \cite{banerjee}.
\begin{lem}[\cite{AudFioVit24}*{Lemma 3.2}]\label{CACCIO}
Let $N\ge1$, $a\in\R$ and $A$ satisfying \eqref{eq:UnifEll}. Let $f\in L^2(Q_1^+,y^a)$, $F\in L^2(Q_1,y^a)^{N+1}$, and let $u$ be a weak solution to \eqref{eq:1}. Then, there exists $C > 0$ depending only on $N$, $a$, $\lambda$ and $\Lambda$ such that for every $\frac{1}{2}\le r' < r < 1$ there holds
\begin{align}\label{caccioppoli.inequality}
        \esssup_{t\in(-r'^2,r'^2)}\int_{B_{r'}^+}y^a u^2+\int_{Q_{r'}^+}y^a |\nabla u|^2 \le C\left[    \frac{1}{(r-r')^2}\int_{Q_r^+}y^a u^2+\|f\|_{L^2(Q_1^+,y^a)}^2+ \|F\|_{L^2(Q_1^+,y^a)}^2\right].
\end{align}
\end{lem}
Finally, we state the main theorem in \cite{AudFioVit24}.
\begin{teo}[\cite{AudFioVit24}*{Theorem 1.1}]\label{teo-C^1,alpha}
    Let $N\ge1$, $a>-1$, $r\in(0,1)$, $p>{N+3+a^+}$ and $\alpha\in(0,1)\cap(0,1-\frac{N+3+a^+}{p}]$. Let $A\in C^{0,\alpha}_p(Q_1^+)$ satisfying \eqref{eq:UnifEll}, $f\in L^p(Q_1^+,y^a )$, $F\in C^{0,\alpha}_p(Q_1^+)$ and let $u$ be a weak solution to \eqref{eq:1}. Then, there exists $C>0$ depending only on $N$, $a$, $\lambda$, $\Lambda$, $r$, $p$, $\alpha$ and $\|A\|_{C^{0,\alpha}_p(Q_1^+)}$ such that
    $$
    \|u\|_{C^{1,\alpha}_p(Q_{r}^+)}\le C\Big(
    \|u\|_{L^2(Q_1^+,y^a)}+
    \|f\|_{L^p(Q_1^+,y^a)}+
    \|F\|_{C^{0,\alpha}_p(Q_1)}
    \Big).
    $$
    Moreover $u$ satisfies the conormal boundary condition
    $$(A\nabla u+F)\cdot e_{N+1}=0\qquad{\rm on  }\hspace{0.1cm}\partial^0Q_{r}^+.$$
\end{teo}

\subsection{Technical results}
In what follows we prove some auxiliary results that we will use throughout the paper.
We begin with a local $L^2$ bound for difference quotients of weak solutions w.r.t. the time variable.
\begin{lem}\label{caccioppoli.della.t}
Let $N\ge1$, $a>-1$ and let $A$ satisfying \eqref{eq:UnifEll} such that $\partial_t A\in {L^\infty(Q_1^+)}$. Let $f\in L^2(Q_1^+,y^a)$ and $F\in L^2(Q_1^+,y^a)^{N+1}$ such that $\partial_t f\in L^2(Q_1^+,y^a)$ and $\partial_t F\in L^2(Q_1^+,y^a)^{N+1}$, 
and let $u$ be a weak solution to \eqref{eq:1}. Consider the difference quotient of $u$ w.r.t. to $t$:
\begin{equation}\label{eq:difference:quotient}
u^h(z,t):=\frac{u(z,t+h)-u(z,t)}{h}, \qquad h > 0.
\end{equation}
Then, there exists $C > 0$ depending only on $N$, $a$, $ \lambda$ and $\Lambda$ such that, for every $r',r\in\mathbb{R}$ satisfying $\frac{1}{2}\le r'<r<1$ and $h > 0$, there holds
\begin{align}\label{caccioppoli.in.difference}
    \begin{aligned}
        \int_{Q_{r'}^+}y^a (u^h)^2&\le C\Big(\frac{1}{(r-r')^2}\int_{Q_r^+}y^a |\nabla u|^2 +
        \|f\|^2_{L^2(Q_1^+,y^a)}+\|F\|^2_{L^2(Q_1^+,y^a)}\\
        &\qquad+
        \|\partial_t A\|_{L^\infty(Q_1^+)}^2\int_{Q_r^+}y^a|\nabla u|^2
        +\|\partial_t f\|^2_{L^2(Q_1^+,y^a)}
        +\|\partial_t F\|^2_{L^2(Q_1^+,y^a)}
        \Big).
    \end{aligned}
\end{align}
\end{lem}

\begin{proof} Fix $r,r'$ such that $\frac{1}{2}\le r'<r<1$. For $h>0$, such that $r<1-h$, let us consider the Steklov average of $u$
    $$u_h(z,t)=\frac{1}{h}\int_t^{t+h}u(z,s)dz,$$
which, by definition, satisfies $\partial_t u_h=u^h$ a.e. in $Q_1$ and the equation
\begin{equation}\label{solution.steklov}
    \int_{Q_r^+}y^a (\partial_t u_h \phi + (A\nabla u)_h\cdot\nabla\phi) = \int_{Q_r^+}y^a (f_h\phi-F_h\cdot\nabla\phi)dzdt, \qquad \forall \phi \in C_c^\infty(Q_1^+).
\end{equation}
Now, for simplicity of the exposition, we assume $f=0$, $F=0$, and we discuss how treat the general case in a second step.

Let us take $\phi=\eta^2 u^h$ as test function in \eqref{solution.steklov}, where $\eta$ is a smooth cut-off function which will define later. Using the Hölder and Young inequalities, the properties of Steklov averages and \eqref{eq:UnifEll}, we obtain
    \begin{align}\label{cacc.eq.t}
    \begin{aligned}
       & \int_{Q_1^+}y^a \eta^2(u^h)^2=\int_{Q_1^+}y^a\left(\eta^2(A\nabla u)_h\cdot\nabla u^h+2\eta u^h(A\nabla u)_h\cdot\nabla\eta  \right)\\
        &\le \left(\int_{Q_1^+} y^a\eta^2
    |(A\nabla u)_h|^2\right)^{1/2}\left(\int_{Q_1^+}y^a \eta^2|\nabla u^h|^2\right)^{1/2}
    +
        2\left(\int_{Q_1^+} y^a
        \eta^2(u^h)^2
        \right)^{1/2}\left(\int_{Q_1^+}y^a |(A\nabla u)_h|^2|\nabla \eta|^2
        \right)^{1/2}\\
        &	\le \frac{C}{\delta}\int_{Q_1^+}y^a |\nabla u|^2+{\delta}\int_{Q_1^+}y^a \eta^2|\nabla u^h|^2
        +\frac{1}{2}\int_{Q_1^+}y^a \eta^2(u^h)^2+C\int_{Q_1^+}y^a |\nabla \eta|^2|\nabla u|^2,
    \end{aligned}
    \end{align}
    for any fixed $\delta>0$ and $C > 0$ depending only on $N$, $a$, $ \lambda$ and $\Lambda$.

In the spirit of \cite[Lemma 3.3]{DongKim11}, we set
    $$
    r_0=r',\quad r_n=r'+\sum_{k=1}^n\frac{r-r'}{2^k},\quad s_n=\frac{r_n+r_{n+1}}{2},\qquad n\in\N,
    $$
    and notice that $r_n$ and $s_n$ are increasing sequences satisfying $r_n < s_n < r_{n+1}$, $r_n\to r$ and $s_n\to r$.

For a given $n\in\mathbb{N}$, taking a cut-off function $\eta_n\in C_c^\infty(Q_1^+)$ in \eqref{cacc.eq.t} such that
    $$
    \supp{\eta_n}\subset Q_{s_n}^+, \quad \eta_n\equiv1 \quad \text{in } Q_{r_n}^+, \quad 0\le\eta_n\le1, \quad |\nabla\eta_n|\le C\frac{2^n}{r-r'},
    $$
we deduce
    \begin{equation}\label{cacc.eq.t.2}
       \frac{1}{2} \int_{Q_{r_n}^+}y^a (u^h)^2
       \le \delta \int_{Q_{s_n}^+}y^a |\nabla u^h|^2
       +C\left(\frac{2^{2n}}{(r-r')^2}+\frac{1}{\delta}\right)\int_{Q_r^+}y^a|\nabla u|^2.
    \end{equation}
    Now, noticing that $u^h$ is a weak solution to
    \begin{equation*}
    y^a\partial_t u^h-\div(y^a A\nabla u^h)=\div(y^a A^h \nabla u)\quad \text{in } Q_r^+,
    \end{equation*}
we may apply the Caccioppoli inequality \eqref{caccioppoli.inequality} to $u^h$, to obtain 
\begin{equation}\label{eq:caccioppoli:t:steklov}
\int_{Q_{s_n}^+}y^a |\nabla u^h|^2\le \frac{C'2^{2n}}{(r-r')^2}
\int_{Q_{r_{n+1}}^+}y^a (u^h)^2+ C'\int_{Q_r^+}y^a |A^h\nabla u|^2,
\end{equation}
for some $C'>0$ independent of $h,r,r'$. Then, setting $\delta=\frac{1}{9}\frac{(r-r')^2}{C'2^{2n}}$ in \eqref{cacc.eq.t.2} and using \eqref{eq:caccioppoli:t:steklov}, it follows    
    $$
    \int_{Q_{r_n}^+}y^a (u^h)^2 
    \le\frac{1}{9}\int_{Q_{r_{n+1}}^+}y^a (u^h)^2
    +\frac{C2^{2n}}{(r-r')^2}\int_{Q_r^+}y^a|\nabla u|^2
    +\frac{C\|\partial_t A\|_{L^\infty(Q_1^+)}^2}{2^{2n}}\int_{Q_r^+}y^a|\nabla u|^2.
    $$
Now, multiplying both sides by $3^{-2n}$ and summing over $n$, we see that  
    \begin{align*}
        &\sum_{n=0}^\infty3^{-2n}\int_{Q_{r_n}^+}y^a (u^h)^2
        \le\sum_{n=0}^\infty 3^{-2n-2}\int_{Q_{r_{n+1}}^+}y^a (u^h)^2\\
        &+\frac{C}{(r-r')^2}\sum_{n=0}^\infty\left(\frac{2}{3}\right)^{2n}\int_{Q_r^+}y^a|\nabla u|^2+\sum_{n=0}^\infty\frac{C\|\partial_t A\|_{L^\infty(Q_1^+)}^2}{6^{2n}}\int_{Q_r^+}y^a|\nabla u|^2,
    \end{align*}
    which implies that
    $$
    \int_{Q_{r'}^+}y^a (u^h)^2\le \frac{C}{(r-r')^2} \int_{Q_r^+}y^a|\nabla u|^2+C\|\partial_t A\|_{L^\infty(Q_1^+)}^2\int_{Q_r^+}y^a|\nabla u|^2,
    $$
    for some new $C > 0$, which is exactly \eqref{caccioppoli.in.difference} in the case $f=0$ and $F=0$.

For non-trivial $f$ and $F$ in the r.h.s., we have two additional terms: one in \eqref{cacc.eq.t} and one in \eqref{eq:caccioppoli:t:steklov}. Both of them can be estimated using the arguments above, namely
\begin{align*}
    &\int_{Q_1^+}y^a(f_h\eta^2 u^h+F_h\cdot\nabla(\eta^2 u^h))\\
    &\le C\|f\|^2_{L^2(Q_1^+,y^a)}+C_\delta\|F\|^2_{L^2(Q_1^+,y^a)}+\frac{1}{4}\int_{Q_1^+}y^a\eta^2 (u^h)^2+\delta\int_{Q_1^+}y^a\eta^2|\nabla u^h|^2,
\end{align*}
for every $\delta>0$, where we have implicitly used that 
\begin{align*}
\int_{Q_1^+} {y^a} \big((f^h)^2+|F^h|^2\big)  
\le \int_{Q_1^+} {y^a}\big((\partial_t f)^2+|\partial_t F|^2\big),
\end{align*}
for every $h \in (0,1)$. With such estimate at hand, the argument above can be slightly adapted to obtain \eqref{caccioppoli.in.difference} in the general case.
\end{proof}
An immediate consequence of the above estimates is that, under suitable regularity assumptions on the data, derivatives (w.r.t. $t$ and $x$) of weak solutions to \eqref{eq:1} are still weak solutions (of a suitable problem of the class \eqref{eq:1}). 
\begin{lem}\label{lemma:derivative:t:solution}
  Let $a>-1$, $N\ge1$, $r\in(0,1)$ and let $A$ satisfying \eqref{eq:UnifEll} such that $\partial_t A \in {L^\infty(Q_1^+)}$. Let $f\in L^2(Q_1^+,y^a)$ and $F\in L^2(Q_1^+,y^a)$ such that $\partial_t f, \partial_t F\in L^2(Q_1^+,y^a)$, and let $u$ be a weak solution to \eqref{eq:1}. Then $v:=\partial_t u$ is a weak solution to
    \begin{equation}\label{eq:solution.t.derivative}
    \begin{cases}
    y^a \partial_t v - {\div }(y^{a} A\nabla v) = y^a \partial_t f + {\div}(y^{a}(\partial_t A\nabla u+ \partial_t F)) \quad &\text{in } Q_r^+,\\
    \displaystyle{\lim_{y\to0^+}} y^a(A\nabla v+\partial_t A\nabla u+ \partial_t F)\cdot e_{N+1}=0              &\text{on }\partial^0Q_r^+.
\end{cases}
\end{equation}
\end{lem}
\begin{proof} Let us fix $0<r<r'<r''<1$ and $h>0$ such that $r''<1-h$. Let $u^h$ be the difference quotient of $u$ w.r.t. to $t$ defined in \eqref{eq:difference:quotient}. By Lemma \ref{caccioppoli.della.t}, $\|u^h\|_{L^2(Q_{r''}^+,y^a)}$ is bounded independently of $h > 0$. Further, since $u^h$ is a weak solution to
 \begin{equation}\label{eq:weak:difference}
    y^a\partial_t u^h-\div(y^a A\nabla u^h)=y^af^h+\div(y^a(F^h +A^h \nabla u))\quad \text{in } Q_{r''}^+,
    \end{equation}
we may use Lemma \ref{CACCIO} to deduce that 
$\|u^h\|_{L^\infty(I_{r'},L^2(B_{r'}^+,y^a))}$ and $\|u^h\|_{L^2(I_{r'},H^1(B_{r'}^+,y^a))}$ are bounded independently of $h > 0$ as well.

Now, let $\xi\in C_c^\infty(B_{r'})$ be a cut-off function such that $0\le\xi\le 1$ and $\xi\equiv1$ in $B_{r}$ and set $v^h:=\xi u^h \in L^2(I_{r'},H^1_0(B_{r'}^+,y^a))$. Arguing as in \cite{AudFioVit24}*{Lemma 4.2, Remark 2.16}, we obtain that $v^h$ is a weak solution to
\begin{equation}\label{eq:xi:u^h}
y^a\partial_t v^h-\div (y^a A\nabla v^h)=y^a\tilde{f}+\div (y^a\tilde{F}),\quad\text{ in }Q_{r'}^+,
\end{equation}
where 
$$
\tilde{f}:=f^h\xi-(F^h+A^h\nabla u)\cdot\nabla\xi-A\nabla u^h\cdot\nabla\xi,
\quad\tilde{F}:=(F^h+A^h\nabla u)\xi-u^hA\nabla\xi,
$$
satisfying also that $\|\partial_t v^h\|_{L^2(I_{r'},H^{-1}(B_{r'}^+,y^a))}\le C$, for some $C>0$ independent of $h > 0$. Consequently,
$$
\|v^h\|_{L^2(I_{r'},H^1_0(B_{r'}^+,y^a))}+\|\partial_t v^h\|_{L^2(I_{r'},H^{-1}(B_{r'}^+,y^a))}\le C,
$$
for some $C > 0$ independent of $h > 0$. Consequently, the Aubin-Lion lemma (see for instance \cite{Simon87}*{Corollary 8}) yields the existence of $v\in {L^2(I_{r'},H^1_0(B_{r'}^+,y^a))}$ such that $v^h\to v$ in $L^2(Q_{r'}^+,y^a)$ and $\nabla v^h\rightharpoonup
\nabla v$ in $L^2(Q_{r'}^+,y^a)$. Since $\xi\equiv 1$ in $Q_r^+$, one has that $u^h\to \partial_t u$ in $L^2(Q_{r}^+,y^a)$ and $\nabla u^h\rightharpoonup
\nabla (\partial_t u)$ in $ {L^2(Q_{r}^+,y^a)}$. Furthermore, by the (H=W) property (see \cite{Zhikov,TerTorVit24}), one has $\partial_t u\in L^2(I_r,H^1(B_r^+,y^a))$ and $\partial_t u\in L^\infty(I_r,L^2(B_r^+,y^a))$ by Fatou's lemma.

Finally, let us fix a test function $\phi\in C_c^\infty(Q_r)$ if $a\in(-1,1)$ or $\phi\in C_c^\infty(Q_r^+)$ if $a\ge1$. By the same argument of \cite{AudFioVit24}*{Lemma 4.2}, we can take the limit as $h\to0^+$ in the weak formulation of \eqref{eq:weak:difference}, to deduce
\begin{align*}
0=&\int_{Q_r^+}y^a\big( -u^h\phi_t +A\nabla u^h\cdot\nabla\phi-f^h\phi+(F^h+A^h\nabla u)\cdot\nabla \phi \big)\\
&\to \int_{Q_r^+}y^a\big( -\partial_t u\phi_t +A\nabla \partial_t u\cdot\nabla\phi-\partial_t f\phi+(\partial_t F+\partial_t A\nabla u)\cdot\nabla \phi \big),
\end{align*}
as $h\to0^+$, that is $\partial_t u$ is a weak solution to \eqref{eq:solution.t.derivative}.
\end{proof}
Analogously, we obtain the equations of the partial derivatives w.r.t. $x$.
\begin{lem}\label{lemma:derivative:i:solution}
  Let $a>-1$, $N\ge1$, $r\in(0,1)$, $i\in\{1,\dots,N\}$ and let $A$ satisfying \eqref{eq:UnifEll} such that $\partial_{x_i} A \in {L^\infty(Q_1^+)}$. Let $f\in L^2(Q_1^+,y^a)$ and $F\in L^2(Q_1^+,y^a)$ such that $\partial_{x_i} f, \partial_{x_i} F\in L^2(Q_1^+,y^a)$, and let $u$ be a weak solution to \eqref{eq:1}. Then $v_i:=\partial_{x_i} u$ is a weak solution to
    \begin{equation}\label{eq:solution.i.derivative}
    \begin{cases}
    y^a \partial_t v_i - {\div }(y^{a} A\nabla v_i) = y^a \partial_{x_i} f + {\div}(y^{a}(\partial_{x_i} A\nabla u+ \partial_{x_i} F)) \quad &\text{in } Q_r^+,\\
    \displaystyle{\lim_{y\to0^+}} y^a(A\nabla v_i+\partial_{x_i} A\nabla u+ \partial_{x_i} F)\cdot e_{N+1}=0              &\text{on }\partial^0Q_r^+.
\end{cases}
\end{equation}
    \end{lem}
    \begin{proof} The proof closely follows the above one and we skip it.
    \end{proof}
The following two auxiliary results are in the spirit of \cite{TerTorVit22}*{Lemma 2.3} and \cite{SirTerVit21a}*{Theorem 7.5} (see also \cite{TerTorVit22}*{Lemma 2.4, Remark 2.5}) in the elliptic setting and turn out to be crucial in rest of the paper.
\begin{lem}\label{lemma-2.3-TTV}
    Let $k\in\mathbb{N}$ and let $v\in C^{k+1,\alpha}_p(Q_1^+)$ such that $v(x,0,t)\equiv0$. Then ${v}/{y}\in C^{k,\alpha}_p(Q_1^+)$ and $[v/y]_{C^{k,\alpha}_p(Q_1^+)} \le [v]_{C^{k+1,\alpha}_p(Q_1^+)}$.
\end{lem}
\begin{proof} The proof follows its elliptic counterpart and we skip it.
\end{proof}
\begin{lem}\label{lemma-2.4-TTV}
    Let $a>-1$, $k\in\mathbb{N}$, $\alpha\in(0,1)$ and let $g\in C^{k,\alpha}_p(Q_1^+)$. Then  the function
    $$\varphi(x,y,t)=\frac{1}{y^{1+a}}\int_0^y s^ag(x,s,t)ds$$
    belongs to $ C^{k,\alpha}_p(Q_1^+)$ and  $[\varphi]_{C^{k,\alpha}_p(Q_1^+)} \le C [g]_{C^{k+1,\alpha}_p(Q_1^+)}$, for some $C>0$ depending only on $a$.
    Moreover, the function
    $$\psi(x,y,t)=\frac{1}{y^{a}}\int_0^y s^ag(x,s,t)ds$$
    satisfies $\partial_y \psi \in C^{k,\alpha}_p(Q_1^+)$ and $[\partial_y \psi]_{C^{k,\alpha}_p(Q_1^+)} \le C [g]_{C^{k+1,\alpha}_p(Q_1^+)}$, for some $C>0$ depending only on $a$.
\end{lem}

\begin{proof}
First, notice that the second statement follows immediately from the first since $\partial_y \psi=-a\varphi+g.$

We prove the first statement by induction. Let $k=0$ and $g\in C^{0,\alpha}_p(Q_1^+)$. The parabolic H\"older continuity in $x$ and $t$ is trivially verified. Indeed, let $P_1=(x_1,y,t_1)$ and $P_2=(x_2,y,t_2)$, then
\begin{align*}
|\varphi(P_2)-\varphi(P_1)|\le \frac{1}{y^{1+\alpha}}\int_0^y s^a|g(x_1,s,t_1)-g(x_2,s,t_2)|ds\le \frac{[g]_{ C^{0,\alpha}_p(Q_1^+)}}{1+a} d_p(P_2,P_1)^\alpha.
\end{align*}
For $\delta>0$, let us consider
$$S_1:=\{(y_1,y_2): 0 < y_1< y_2\le 1,\text{ and } y_2-y_1\ge\delta y_2\},$$
$$S_2:=\{(y_1,y_2): 0< y_1< y_2\le 1,\text{ and } y_2-y_1<\delta y_2\}.$$
Taking $y_1,y_2\in S_1$, one has
\begin{align*}
|\varphi(x,y_2,t)-\varphi(x,y_1,t)| &= \left|\frac{1}{y_2^{a+1}}\int_0^{y_2}s^ag(x,s,t)ds-
\frac{1}{y_1^{a+1}}\int_0^{y_1}s^ag(x,s,t)ds
\right|\\
&=\left|\frac{1}{y_2^{a+1}}\int_0^{y_2}s^a(g(x,s,t)-g(x,0,t))ds-
\frac{1}{y_1^{a+1}}\int_0^{y_1}s^a(g(x,s,t)-g(x,0,t))ds
\right|\\
&\le \frac{[g]_{C^{0,\alpha}_p(Q_1^+)}}{y_2^{a+1}}\int_0^{y_2}s^{a+\alpha}ds+\frac{[g]_{C^{0,\alpha}_p(Q_1^+)}}{y_1^{a+1}}\int_0^{y_1}s^{a+\alpha}ds=\frac{[g]_{C^{0,\alpha}_p(Q_1^+)}}{a+\alpha+1}(y_2^\alpha+y_1^\alpha)\\
&\le \frac{ 2[g]_{C^{0,\alpha}_p(Q_1^+)}}{a+\alpha+1} y_2^\alpha\le \frac{2[g]_{C^{0,\alpha}_p(Q_1^+)}}{\delta^\alpha(a+\alpha+1)}(y_2-y_1)^\alpha.
\end{align*}
Let now $y_1,y_2\in S_2$. Then, 
\begin{align*}
|\varphi(x,y_2,t)-\varphi(x,y_1,t)| &=\left|\frac{1}{y_2^{a+1}}\int_0^{y_2}s^a(g(x,s,t)-g(x,0,t))ds-
\frac{1}{y_1^{a+1}}\int_0^{y_1}s^a(g(x,s,t)-g(x,0,t))ds
\right|\\
&\le\frac{1}{y_2^{a+1}}\int_{y_1}^{y_2}s^a|g(x,s,t)-g(x,0,t)|ds
+\left(\frac{1}{y_1^{a+1}}-\frac{1}{y_2^{a+1}}\right)\int_{0}^{y_1}s^a|g(x,s,t)-g(x,0,t)|ds\\
&\le \frac{[g]_{C^{0,\alpha}_p(Q_1^+)}}{a+\alpha+1}\left(
\frac{y_2^{a+\alpha+1}-y_1^{a+\alpha+1}}{y_2^{a+1}}+\left(\frac{1}{y_1^{a+1}}-\frac{1}{y_2^{a+1}}\right)y_1^{a+\alpha+1}\right)\\
&=\frac{[g]_{C^{0,\alpha}_p(Q_1^+)}}{a+\alpha+1} 
\left(y_2^\alpha-	
y_1^\alpha\left(\frac{y_1}{y_2}\right)^{a+1}
+y_1^\alpha\left( 1- \left(\frac{y_1}{y_2}\right)^{a+1}\right)
\right)\\
&=\frac{[g]_{C^{0,\alpha}_p(Q_1^+)}}{a+\alpha+1} \left(
y_2^\alpha+y_1^\alpha-2y_1^\alpha\left(\frac{y_1}{y_2}\right)^{a+1}\right) \\
&=\frac{[g]_{C^{0,\alpha}_p(Q_1^+)}}{a+\alpha+1} \left(
y_2^\alpha-y_1^\alpha+2y_1^\alpha \left(1-\left(\frac{y_1}{y_2}\right)^{a+1}\right)\right)\\
&\le \frac{[g]_{C^{0,\alpha}_p(Q_1^+)}}{a+\alpha+1} \left(
y_2^\alpha-y_1^\alpha+C_a y_1^\alpha \left(1-\frac{y_1}{y_2}\right)\right),
\end{align*}
where $C_a>0$ is a constant which depends only on $a$. Consequently, since by definition $y_1/y_2 > 1-\delta$, we have
\begin{align*}
&\frac{|\varphi(x,y_2,t)-\varphi(x,y_1,t)|}{(y_2-y_1)^\alpha}\le \frac{[g]_{C^{0,\alpha}_p(Q_1^+)}}{a+\alpha+1} \left(
\frac{y_2^\alpha-y_1^\alpha}{(y_2-y_1)^\alpha}+C_a\frac{y_1^\alpha(y_2-y_1)}{y_2(y_2-y_1)^\alpha}
\right)\le \frac{[g]_{C^{0,\alpha}_p(Q_1^+)}}{a+\alpha+1} \left(
1+C_a\delta^{1-\alpha} \right),
\end{align*}
and hence, the case $k=0$ follows.

Next, let us assume that the our claim is true for some $k \in \N$ and let us prove it for
$k+1$: we assume $g\in C^{k+1,\alpha}_p(Q_1^+)$ and show that $\varphi\in  C^{k+1,\alpha}_p(Q_1^+)$.

Since
$$
\partial_{x_i}\varphi =\frac{1}{y^{1+a}}\int_0^y s^a\partial_{x_i}g(x,s,t)ds, \quad i = 1,\dots,N, \qquad \partial_{t}\varphi =\frac{1}{y^{1+a}}\int_0^y s^a\partial_{t}g(x,s,t)ds,
$$
we immediately have that $\varphi$ is $C^{k+1}$ in $x$ and $t$. Moreover, the boundedness of the $C_t^{\frac{1+\alpha}{2}}$-seminorm of the mixed-derivates follows as the case $k = 0$. 

We are left to prove that $\partial_y\varphi\in C^{k,\alpha}_p(Q_1^+)$. To do this, we can rewrite $\varphi$ as
$$\varphi(x,y,t)=\frac{1}{y^{1+a}}\int_0^y s^a(g(x,s,t)-g(x,0,t))ds+\frac{g(x,0,t)}{a+1},$$
and observe that
$$\partial_y \varphi(x,y,t)=-\frac{a+1}{y^{2+a}}\int_0^y s^{a+1}\frac{g(x,s,t)-g(x,0,t)}{s}ds+\frac{g(x,y,t)-g(x,0,t)}{y}.$$
By Lemma \ref{lemma-2.3-TTV}, one has that $\frac{g(x,y,t)-g(x,0,t)}{y}\in C^{k,\alpha}_p(Q_1^+)$ and our claim follows by the inductive assumption.
\end{proof}
%

%
%
%

%
%

%
%
%

%
%
%
%
\section{Liouville theorem}
This section is devoted to the proof of the Liouville-type Theorem \ref{teo:polynomial:liouville}. We remark that, in the case $a \in (-1,1)$, entire solutions to \eqref{eq-liouville-polinomiale} satisfy the smoothness estimates in \cite[Theorem 1.1]{BanGar23}, hence, the proof of the Liouville theorem follows by a standard rescaling argument (for example, see \cite[Proposition 1.19]{XX22}).
\begin{proof}[Proof of Theorem \ref{teo:polynomial:liouville}] Let us fix $R>1$ and define 
     $$
     \tilde{\gamma}:=a^++2\gamma+N+3.
     $$ 

     \emph{Step 1.} 
     Choosing $r'=R$ and $r=2R$ in \eqref{caccioppoli.inequality} and using \eqref{growth-polinomiale}, we get 
     \begin{equation}\label{eq:liouville:caccioppoli:x}
           \int_{Q_R^+}y^a|\D u|^2 \le \frac{C}{R^2} \int_{Q_{2R}^+}y^a u^2 \le C R^{\tilde{\gamma}-2},
     \end{equation}
     for some $C>0$ depending only on $N$ and $a$. On the other hand, choosing $r'=R$ and $r=2R$ in \eqref{caccioppoli.in.difference} and combining \eqref{eq:liouville:caccioppoli:x} and \eqref{growth-polinomiale}, we obtain 
     \begin{equation}\label{eq:liouville:caccioppoli:t}
           \int_{Q_R^+}y^a (\partial_t u)^2  \le \frac{C}{R^4} \int_{Q_{4R}^+}y^a u^2 \le C R^{\tilde{\gamma}-4},
     \end{equation}
for some new $C>0$.

\smallskip

\emph{Step 2.} In this step we prove that $u$ is a polynomial in $x$. By Lemma \ref{lemma:derivative:i:solution}, for every multiindex $\beta\in \N^N$, $\partial_{x}^\beta u$ is a weak solution to \eqref{eq-liouville-polinomiale}. Then, by iterating \eqref{eq:liouville:caccioppoli:x}, one has
\[
\int_{Q_R}y^a(\partial_{x}^{\beta} u)^2\le \int_{Q_R}y^a|\D u|^2\le C R^{\tilde{\gamma}-2| \beta|}.
\]
Consequently, taking $\beta$ such that $\tilde{\gamma}-2| \beta|<0$ and passing to the limit as $R\to+\infty$, it follows $\partial_{x}^{\beta} u=0$ and therefore $u$ is a polynomial in the variable $x$, with degree less or equal than $m$ (the bound on the degree immediately follows by \eqref{growth-polinomiale}).

\smallskip

\emph{Step 3.} A slight modification of the above argument, which uses \eqref{eq:liouville:caccioppoli:t} instead of \eqref{eq:liouville:caccioppoli:x}, shows that $u$ is a polynomial in the variable $t$, with degree less or equal than $\lfloor \frac{m}{2}\rfloor$.

\smallskip

\emph{Step 4.} The last step is to prove that $u$ is polynomial in $y$. By \cite{AudFioVit24}*{Remark 4.4}, we notice that the even extension of $u$ w.r.t. $y$ is an entire solution to 
\begin{equation}\label{eq:liouville:y}
|y|^a\partial_t u-{\div}(|y|^a \nabla u)=0\quad\text{in }\mathbb{R}^{N+1}\times\mathbb{R}.
\end{equation}
Further, by \cite{AudFioVit24}*{Lemma 5.2}, $v:=|y|^a\partial_y u$ is an entire solution to 
\begin{equation*}
|y|^{-a}\partial_t u-{\div}(|y|^{-a} \nabla u)=0\quad\text{in }\mathbb{R}^{N+1}\times\mathbb{R},
\end{equation*}
while 
\begin{equation}\label{eq:liouville:w_1}
w_1:=|y|^{-a}\partial_y v=\partial_{yy}u-a\frac{\partial_y u}{y},
\end{equation}
is an entire solution to \eqref{eq:liouville:y}. Now, applying \eqref{caccioppoli.inequality} twice, we deduce
\begin{equation}\label{eq:inequality:w_1}
    \int_{{Q}_R}|y|^a w_1^2 \leq \int_{{Q}_R}|y|^{-a}|\nabla v|^2  \le \frac{C}{R^2}\int_{{Q}_{2R}}|y|^{-a}v^2
    \le  \frac{C}{R^2}\int_{{Q}_{2R}}|y|^{a}|\nabla u|^2\le \frac{C}{R^4}\int_{{Q}_{4R}}|y|^{a}u^2\le C R^{\tilde{\gamma}-4}.
\end{equation}
Setting 
\begin{equation}\label{eq:liouville:w_j}
    w_{j+1}:=\partial_{yy}w_{j}+a\frac{\partial_y w_{j}}{y},
\end{equation}
and noticing that $w_{j+1}$ is an entire solution to \eqref{eq:liouville:y} for $j\in\mathbb{N}_+$, we may iterate the argument above to show the existence of $k\in\mathbb{N}$ such that $\tilde{\gamma}-4{k}<0$ and 
$$
\int_{{Q}_R}|y|^a w_{{k}}\le C R^{\tilde{\gamma}-4{k}}.
$$
Hence, taking the limit as $R\to+\infty$, we obtain $w_k = 0$, that is
$$\partial_{yy}w_{{k-1}}+a\frac{\partial_y w_{{k-1}}}{y}=0.$$
The above ODE can be explicitly solved:
\begin{equation}\label{eq:ODESolStepk-1}
w_{k-1} = c_{2k-1}(x,t)y|y|^{-a}+c_{2k-2}(x,t),
\end{equation}
where $c_{2k-1}(x,t)$ and $c_{2k-2}(x,t)$ are polynomials. Now, iteratively solving the ODEs in \eqref{eq:liouville:w_j} and \eqref{eq:liouville:w_1}, we obtain an explicit formula for $u$: 
\begin{equation}\label{eq:all:entire:sol}
u=c_0(x,t)+\sum_{i\ge 1}y^{2i}c_{2i}(x,t)+\sum_{i\ge 1}y^{2i-1}|y|^{-a}c_{2i-1}(x,t),
\end{equation}
where $c_i(x,t)$ are polynomial. All solutions to \eqref{eq:liouville:y} satisfying a polynomial growth condition (without imposing any symmetry condition) have the form \eqref{eq:all:entire:sol}. Since $u$ is an even solution (which comes from the conormal condition at the hyperplane), $c_{2i-1}\equiv0$ for every $i\ge1$. Therefore, our statement follows from the growth assumption \eqref{growth-polinomiale}.
\end{proof}

In the following remark, we also provide a classification of the entire solutions to \eqref{eq-liouville-polinomiale} satisfying the growth condition \eqref{growth-polinomiale}. Such classification was already obtained in \cite[Lemma 3.2]{BanDanGarPet21} in the range $a\in (-1,1)$ (see also \cite{GarRos}*{Lemma 5.2} in the elliptic setting). We present the proof for completeness.

\begin{oss}
Let $a>-1$ and let $q_\kappa = q_\kappa(x,t)$ be a polynomial of parabolic degree $\kappa$ in $\R^{N}\times\R$. Then, there exists a unique polynomial $\tilde{q}_\kappa = \tilde{q}_\kappa(x,y,t)$ of parabolic degree $\kappa$ in $\R^{N}\times\R_+\times\R$ such that $\tilde{q}_\kappa$ satisfies \eqref{eq-liouville-polinomiale} and $\tilde{q}_\kappa(x,0,t)=q_\kappa(x,t)$ for every $(x,t) \in \R^{N}\times\R$. Moreover, 
\begin{equation}\label{eq:classification:formula}
\tilde{q}_\kappa(x,y,t)={q}_\kappa(x,t)+\sum_{i=1}^{\lfloor \kappa/2 \rfloor}\frac{y^{2i}}{2i!}c_{2i}(\partial_t -\Delta_{x})^i q_\kappa(x,t),\quad\text{where } c_{2i}=\prod_{j=1}^{i}\frac{2j-1}{2j-1+a}.
\end{equation}
\begin{proof}[Proof of \eqref{eq:classification:formula}]
We denote with $\Delta_{(x,y)}$ the Laplacian in the variables $(x,y)$, $\Delta_{x}$ the Laplacian in the variable $x$ and $(\partial_t-\Delta_{x})^i$ the heat operator applied $i$ times.  Let $M:=\lfloor \kappa/2 \rfloor$.

If such a polynomial $\tilde{q}_\kappa$ exists and satisfies the Neumann boundary condition $\lim_{y\to0^+}y^a \partial_y \tilde{q}_\kappa=0$, then
\begin{equation}\label{eq:clas:q:kappa}
\tilde{q}_\kappa(x,y,t)=q_\kappa(x,t)+\sum_{i=1}^{M}y^{2i}q_i(x,t),
\end{equation}
where $q_i(x,t)$ are polynomials such that $y^{2i} q_i(x,t)$ have parabolic degree at most $\kappa$.
Indeed, according to Theorem \ref{teo-C^1,alpha}, $\tilde{q}_k$ satisfies the stronger Neumann boundary condition $\lim_{y\to0^+} \partial_y \tilde{q}_\kappa=0$. This implies that $\tilde{q}_\kappa$ cannot contain a nontrivial term $yq_1(x,t)$. As in the proof of Theorem \ref{teo:polynomial:liouville}, we may iterate this argument to show that any term of the form $y^{2i+1}q_i(x,t)$ is identically zero.

Notice that
\begin{align}\label{eq:classification}
\begin{aligned}
&0=(\partial_t -\Delta_{(x,y)})\tilde{q}_\kappa-\frac{a}{y}\partial_y \tilde{q}_\kappa\\
&=(\partial_t -\Delta_x )q_\kappa +\sum_{i=1}^{M}y^{2i}(\partial_t -\Delta_x )q_i-\sum_{i=1}^{M}2i(2i-1)y^{2i-2}q_i-a\sum_{i=1}^{M} 2iy^{2i-2}q_i\\
&=(\partial_t-\Delta_x)q_\kappa-(2+2a)q_1+
\sum_{i=1}^{M-1}y^{2i}\big((\partial_t-\Delta_x) q_i-(2i+2)(2i+1+a)q_{i+1}\big)
+y^{2M}( \partial_t-\Delta_x )q_M.
\end{aligned}
\end{align}
Now, by iteratively solving  the equation in \eqref{eq:classification} we obtain
\begin{align}\label{eq:explicit:classification}
\begin{aligned}
&q_1=\frac{(\partial_t -\Delta_{x}){q}_\kappa}{2(1+a)},\\
&q_2= \frac{(\partial_t -\Delta_{x})q_1}{4(3+a)}=\frac{(\partial_t -\Delta_{x})^2q_\kappa}{4(3+a)2(1+a)},\\
&q_i=(\partial_t -\Delta_{x})^i q_\kappa\prod_{j=1}^{i}\frac{1}{2j(2j-1+a)}=\frac{(\partial_t -\Delta_{x})^i q_\kappa}{2i!}\prod_{j=1}^{i}\frac{2j-1}{2j-1+a},\qquad \text{ for } i\in\{1,\dots,M\}.
\end{aligned}
\end{align}
By construction, the function $\tilde{q}_\kappa$ defined in \eqref{eq:classification:formula} satisfies our statement. The uniqueness of $\tilde{q}_\kappa$ immediately follows by the explicit formula \eqref{eq:explicit:classification} and the linearity of
the differential operator.
\end{proof}
\end{oss}

%
%
\section{$C^{2,\alpha}_p$ regularity}\label{section:c^2}

The goal of this section is to prove Theorem \ref{teo1} when $k=0$; that is, the following 
\begin{teo}\label{teo1-C2}
    Let $N\ge1$, $a>-1$, $r\in(0,1)$, $\alpha\in(0,1)$. Let $A\in C^{1,\alpha}_p(Q_1^+)$ satisfying \eqref{eq:UnifEll}, $f\in C^{0,\alpha}_p(Q_1^+)$ and $F\in C^{1,\alpha}_p(Q_1^+)$ and let $u$ be a weak solution to \eqref{eq:1}. Then, there exists $C>0$ depending only on $N$, $a$, $\lambda$, $\Lambda$, $r$, $\alpha$ and $\|A\|_{C^{1,\alpha}_p(Q_1^+)}$ such that
    \begin{equation}\label{eq:c2}
    \|u\|_{C^{2,\alpha}_p(Q_{r}^+)}\le C\Big(
    \|u\|_{L^2(Q_1^+,y^a)}+
    \|f\|_{C^{0,\alpha}_p(Q_1)}+
    \|F\|_{C^{1,\alpha}_p(Q_1)}
    \Big).
    \end{equation}
\end{teo}
The proof is based on some a priori estimates and an approximation argument we present below.

\subsection{A priori $C^{2,\alpha}_p$ estimates} We begin by showing the a priori $C^{2,\alpha}_p$ estimates, stated in the following
\begin{pro}\label{teo-C^2,alpha}
    Let $N\ge1$, $a>-1$, $\alpha \in (0,1)$ and $r\in(0,1)$. Let $A\in C^{1,\alpha}_p(Q_1^+)$ satisfying \eqref{eq:UnifEll}, $f\in C^{0,\alpha}_p(Q_1^+)$, $F\in C^{1,\alpha}_p(Q_1^+)$ and let $u\in C^{2,\alpha}_p(Q_1^+)$ be a weak solution to \eqref{eq:1}.
Then, there exists $C>0$ depending only on $N$, $a$, $\lambda$, $\Lambda$, $r$, $\alpha$, $\|A\|_{C^{1,\alpha}_p(Q_1^+)}$ such that
\eqref{eq:c2} holds.
\end{pro}
\begin{proof} The proof is divided in several steps as follows.

\smallskip

\emph{Step 1.} Without loss of generality we prove the statement for $r=1/2$. To simplify the notation, let $x_{N+1}=y$, $\partial_i:=\partial_{x_i}$ for $i=1,\dots,N+1$, and $\partial_{ij}:=\partial_i\partial_j$ for $i,j=1,\dots,N+1$. 
In the following, we will refer to the variable $y$ as either $y$ or $x_{N+1}$ depending on what seems more convenient. We begin with some preliminary observations.

\smallskip

\noindent $\bullet$ By the regularity assumptions and Theorem \ref{teo-C^1,alpha}, one has that $u$ satisfies the equation pointwise in $Q_1^+$, and so
\begin{equation}\label{eq:c2:pointwisely}
\Big(\partial_t u-\sum_{i,j=1}^{N+1}A_{i,j}\partial_{ij}u-\frac{a}{y}\sum_{j=1}^{N+1} A_{N+1,j}\partial_j u\Big) = \Big(g+\frac{a}{y}F_{N+1}\Big) \quad \text{ in } Q_1^+,
\end{equation}
where $g := \sum_{i,j=1}^{N+1}\partial_i A_{i,j}\partial_{j}u+f+\sum_{i=1}^{N+1} \partial_i F\in C^{0,\alpha}_p(Q_1^+)$ satisfies 
\begin{align}\label{eq:c2:interpolation:g}
\begin{aligned}
\|g\|_{C^{0,\alpha}_p(Q_1^+)}
&\le 
3\|A\|_{ C^{1,\alpha}_p(Q_1^+)}\|u\|_{ C^{1,\alpha}_p(Q_1^+)}
+\|f\|_{ C^{0,\alpha}_p(Q_1^+)}
+\|F\|_{ C^{1,\alpha}_p(Q_1^+)}\\
&\le 
C\left(\|u\|_{ L^\infty(Q_1^+)}
+\|D^2u\|_{ L^\infty(Q_1^+)}
+\|f\|_{ C^{0,\alpha}_p(Q_1^+)}
+\|F\|_{ C^{1,\alpha}_p(Q_1^+)}\right),
\end{aligned}
\end{align}
for some $C>0$ depending on $\|A\|_{ C^{1,\alpha}_p(Q_1^+)}$, thanks to the interpolation inequality \eqref{lemma:interpolation}.

\smallskip

\noindent $\bullet$ By the regularity assumptions on the data and $u$, and using the conormal boundary condition in \eqref{eq:1} (which is satisfied pointwise by Theorem \ref{teo-C^1,alpha}), we can take the limit as $y\to0^+$ in \eqref{eq:c2:pointwisely} to get
\begin{align}\label{eq:c2:boundary:pointwisely}
\begin{aligned}
\lim_{y	\to0^+}\frac{a\big(\sum_{j=1}^{N+1}A_{N+1,j}\partial_j u+F_{N+1}\big)(x,y,t)}{y}&=a\partial_y \big(\sum_{j=1}^{N+1}A_{N+1,j}\partial_j u+F_{N+1}\big)(x,0,t)\\
&=\Big(  \partial_t u-\sum_{i,j}A_{i,j}\partial_{ij}u -g\Big)(x,0,t),
\end{aligned}
\end{align}
for every $(x,0,t)\in\partial^0 Q_1^+$.

\smallskip

\noindent $\bullet$ It is enough prove that for every $\delta>0$ sufficiently small,
\begin{equation}\label{eq:c2:XX:estimate}
[ u]_{C^{2,\alpha}_p(Q_{1/2}^+)}\le \delta  [ u]_{C^{2,\alpha}_p(Q_{1}^+)}+C_\delta \left(\|D^2u\|_{L^\infty(Q_1^+)}+\|u\|_{L^\infty(Q_1^+)}+\|f\|_{C^{0,\alpha}_p(Q_1^+)}+
    \|F\|_{C^{1,\alpha}_p(Q_1^+)}\right),
\end{equation}
for some $C_\delta>0$ depending only on $\delta$, $N$, $a$, $\lambda$, $\Lambda$, $\alpha$, $\|A\|_{C^{1,\alpha}_p(Q_1^+)}$. We will show later how \eqref{eq:c2} follows by \eqref{eq:c2:XX:estimate}.
    
    \smallskip
        
    \emph{Step 2. Contradiction argument and blow-up sequences.} By contradiction we assume that there exist $\alpha\in(0,1)$, $A^{(k)},F^{(k)} \in C^{1,\alpha}_p(Q_1^+)$, $f_k\in C^{0,\alpha}_p(Q_1^+)$ with $\|A^{(k)}\|_{C^{1,\alpha}_p(Q_1^+)}\le C$ and $u_k\in C^{2,\alpha}_p(Q_1^+)$ such that
    \begin{equation}\label{equation u_k C^2}
        \begin{cases}
        y^a\partial_t u_k-{\div }(y^{a} A^{(k)}\nabla u_k)=y^af_k+{\div}(y^{a} F^{(k)}) &{\rm in  }\hspace{0.1cm}Q_1^+,\\
        \displaystyle{\lim_{y\to0^+}}y^a\left(A^{(k)}\nabla u_k+F^{(k)}\right)\cdot e_{N+1}=0&{\rm on  }\hspace{0.1cm}\partial^0Q_1^+,
    \end{cases}
    \end{equation}
    and there exists a small $\delta_0>0$ such that
    \begin{equation}\label{eq:c2:XX:contradiction}
[ u_k]_{C^{2,\alpha}_p(Q_{1/2}^+)}> \delta_0  [ u_k]_{C^{2,\alpha}_p(Q_{1}^+)}+k \left(\|D^2 u_k\|_{L^\infty(Q_1^+)}+\| u_k\|_{L^\infty(Q_1^+)}+\|f_k\|_{C^{0,\alpha}_p(Q_1^+)}+
    \|F^{(k)}\|_{C^{1,\alpha}_p(Q_1^+)}\right).
\end{equation}
 Let us define
   \begin{align*}
      L_k:=\max\Big\{&\big\{[\partial_{i j}u_k]_{C^{0,\alpha}_p(Q^+_{1/2})}:i,j=1,\dots,N+1\big\},
      [\partial_t u_k]_{C^{0,\alpha}_p(Q^+_{1/2})},\big\{[\partial_{i} u_k]_{C_t^\frac{1+\alpha}{2}(Q^+_{1/2})}i=1,\dots,N+1\big\}\Big\},
     \end{align*}
and distinguish two cases: first, we assume that there exist $i,j\in\{1,\dots,N+1\}$ such that 
\begin{equation}\label{eq:LkChoice}
L_k = [\partial_{i j} u_k]_{C^{0,\alpha}_p(Q^+_{1/2})}.
\end{equation}
Later we will deal with the second case, when $L_k=[\partial_{i} u_k]_{C_t^\frac{1+\alpha}{2}(Q^+_{1/2})}$. The case $L_k=[\partial_t u_k]_{C^{0,\alpha}_p(Q^+_{1/2})}$ is very similar to \eqref{eq:LkChoice} and we skip it.

\smallskip

Now, we consider two sequences of points $P_k(z_k,t_k), \bar{P}_k(\xi_k,\tau_k)\in Q^+_{1/2}$ such that 
$$
\frac{|\partial_{ij} u_k(P_k)-\partial_{ij} u_k(\bar{P}_k)|}{d_p(P_k,\bar{P}_k)^\alpha}\ge \frac{L_k}{2}, 
$$
and define $r_k:=d_p(P_k,\bar{P}_k)$. Notice that it must be $r_k\to0$ as $k\to+\infty$, since
$$ 
\frac{L_k}{2}\le \frac{|\partial_{ij} u_k(P_k)-\partial_{ij} u_k(\bar{P}_k)|}{d_p(P_k,\bar{P}_k)^\alpha}\le 2\frac{\|\partial_{ij} u_k\|_{L^\infty(Q_{1/2}^+)}}{r_k^\alpha} \leq 2\frac{[ u_k]_{C^{2,\alpha}_p(Q_{1/2}^+)}}{r_k^\alpha k} \leq 2\frac{L_k}{r_k^\alpha k},
$$
where we have used \eqref{eq:c2:XX:contradiction} and the definition of $L_k$.

\

Let $\hat{z}_k=(\hat{x}_k,\hat{y}_k)\in B_{1/2}^+$ to be specified below. For $k$ large, let us define 
 $$
 Q(k):=\frac{{B^+_1}-{\hat{z}_k}}{r_k}\times\frac{(-1-t_k,1-t_k)}{r_k^2},
 $$
 and set $Q^\infty:=\lim_{k\to+\infty}Q(k)$, along an appropriate subsequence. For $(z,t)\in Q(k)$, consider the blow-up sequence
\begin{equation}\label{w_k degenerate C^2}
    w_k(z,t):=\frac{u_k(r_k z+\hat{z}_k,r_k^2 t+t_k)-T_k(z,t)}{[u_k]_{C^{2,\alpha}_p(Q_1^+)} r_k^{2+\alpha}},
\end{equation}
where $T_k$ is the quadratic parabolic polynomial
$$
T_k(z,t)=u_k(\hat{z}_k,t_k)+r_k\sum_{i=1}^{N+1} \partial_i u_k(\hat{z}_k,t_k) x_i+\frac{r_k^2}{2}\sum_{i,j=1}^{N+1} \partial_{ij} u_k (\hat{z}_k,t_k) x_ix_j+r_k^2\partial_t u_k (\hat{z}_k,t_k) t.
$$
Notice that $w_k$ satisfies
\begin{equation}\label{eq:c2:fix0}
w_k(0)=|\D w_k(0)|=|D^2 w_k(0)|=\partial_t w_k(0)=0.
\end{equation}
At this point we distinguish two cases:

\smallskip

\noindent \textbf{Case 1:} 
$$\frac{y_k}{r_k}=\frac{d_p(P_k,\Sigma)}{r_k}\to+\infty, \quad \text{ as } k\to\infty.$$
In this case we set $\hat{z}_k=z_k$ and we have $Q^\infty=\mathbb{R}^{N+2}$.

\smallskip

\noindent \textbf{Case 2:} $$\frac{y_k}{r_k}=\frac{d_p(P_k,\Sigma)}{r_k}\le C,$$
for some $C > 0$ independent of $k$. In this case we set $\hat{z}_k=(x_k,0)$ and we have $Q^\infty=\R_+^{N+1}\times\R$.

\smallskip
    
    \emph{Step 3. H\"older estimates and convergence of the blow-up sequences.} Let us fix a compact set $K\subset Q^\infty$. Then, 
    $K\subset Q(k)$ for any $k$ large enough. By definition of the $C^{0,\alpha}_p$ seminorm and the parabolic scaling, for every $P=(z,t)$, $Q=(\xi,\tau)\in K$ and $i,j\in\{1,\dots,N+1\}$, we have
    \begin{align*}
    |\partial_{ij} w_k(P)-\partial_{ij} w_k(Q)|\le \frac{|\partial_{ij}u_k(r_kz+\hat{z}_k,r_k^2 t+t_k)-\partial_{ij}u_k (r_k \xi+\hat{z}_k,r_k^2 \tau+t_k)|}{[u_k]_{C^{2,\alpha}_p(Q_1^+)}r_k^\alpha}\le d_p(P,Q)^\alpha,
\end{align*}
and thus
\begin{equation}\label{seminorm C^2}
    \sup_{\substack{P,Q\in K\\P\not=Q}}\frac{ |\partial_{ij} w_k(P)-\partial_{ij} w_k(Q)|}{d_p(P,Q)^\alpha}\le 1.
\end{equation}
In a similar way, it is not difficult to obtain
\begin{equation}\label{seminorm C^2 t}
    \sup_{\substack{P,Q\in K\\P\not=Q}}\frac{ |\partial_{t} w_k(P)-\partial_{t} w_k(Q)|}{d_p(P,Q)^\alpha}\le 1.
\end{equation}
Further, for every $(z,t),(z,\tau)\in K$ and $i\in\{1,\dots,N+1\}$, there holds
\begin{align*}
    | \partial_i w_k(z,t)- \partial_i w_k(z,\tau)|\le \frac{|\partial_i u_k(r_k z+\hat{z}_k,r_k^2 t+t_k)-\partial_iu_k(r_k z+\hat{z}_k,r_k^2 \tau+t_k)|}{[u_k]_{C^{2,\alpha}_p(Q_1^+)}r_k^{1+\alpha}}
    \le |t-\tau|^\frac{1+\alpha}{2},
\end{align*}
which implies
\begin{equation}\label{seminorm C^2 time}
    \sup_{\substack{(z,t),(z,\tau)\in K\\t\not=\tau}}\frac{|\partial_{i}w_k(z,t)-\partial_{i} w_k(z,\tau)|}{|t-\tau|^{\frac{1+\alpha}{2}}}\le 1.
\end{equation}
Combining \eqref{seminorm C^2}, \eqref{seminorm C^2 t} and \eqref{seminorm C^2 time}, we deduce that $[w_k]_{C^{2,\alpha}_p(K)}$ is uniformly bounded in $k$, for every compact set $K\subset Q^\infty$ (notice that the estimates above are valid in both \textbf{Case 1} and \textbf{Case 2}, by definition of $Q^\infty$). Consequently, in light of \eqref{eq:c2:fix0}, $\|w_k\|_{C^{2,\alpha}_p(K)}$ is uniformly bounded as well, and so we may apply the Arzel\`a-Ascoli theorem to conclude that $w_k\to\Bar{w}$ in $C^{2,\gamma}_p(K)$, for every $\gamma \in (0,\alpha)$. Finally, a standard diagonal argument combined with \eqref{seminorm C^2}, \eqref{seminorm C^2 t} and \eqref{seminorm C^2 time}, shows that
\[
w_k \to \Bar{w} \quad \text{ in } C^{2,\gamma}_p(K), \quad \text{ for every } K \subset\subset Q^\infty,
\]
up to passing to a suitable subsequence, and 
\begin{equation}\label{eq:c2:global:holder}
    [\bar{w}]_{C^{2,\alpha}_p(Q^\infty)}\le C_N.
\end{equation}
for some $C_N>0$ which depends only on $N$, by the definition of the $C_p^{2,\alpha}$ seminorm.

\smallskip

    \emph{Step 4.} The next step is to prove that $\partial_{ij}\Bar{w}$ is not constant, where $i,j$ are the indexes fixed in \eqref{eq:LkChoice}. To do this, we consider two sequences of points in $Q(k)$, defined as
$$
S_k=\left(\frac{\xi_k-\hat{z}_k}{r_k},\frac{\tau_k-t_k}{r_k^2}\right), \qquad \Bar{S_k}:=\left(\frac{z_k-\hat{z}_k}{r_k},0\right), \quad k \in \N.
$$
In \textbf{Case 1}, one has $\hat{z}_k=z_k$, then $S_k\to S\in Q^\infty$, up to passing to a subsequence and $\Bar{S}_k=0$ for every $k$. Then, using the definition of $L_k$ and \eqref{eq:c2:XX:contradiction}, it follows
\[
|\partial_{ij}w_k(S_k)-\partial_{ij}w_k(\bar{S}_k)|=|\partial_{ij}{u}_k(\bar{P}_k)-\partial_{ij}u_k(P_k)|\ge \frac{L_k}{2[u_k]_{C^{2,\alpha}_p(Q_1^+)}}\ge C \delta_0,
\]
for some $C > 0$ independent on $k$ and thus, passing to the limit as $k \to +\infty$, we obtain $|\partial_{ij}\Bar{w}({S})-\partial_{ij}\Bar{w}(0)|\ge C \delta_0$, that is, $\partial_{ij}\Bar{w}$ is not constant.

In \textbf{Case 2} we can argue in a similar way: we have $\hat{z}_k = (x_k,0)$ and so $\bar{S}_k = \frac{y_k}{r_k}e_{n+1}$. Recalling that $\frac{y_k}{r_k}$ is uniformly bounded by definition, $S_k \to \bar{S}$, for some $\Bar{S}$, up to passing to a subsequence. On the other hand, the sequence $S_k$ can be written as 
$$
S_k=\left(\frac{\xi_k-{z}_k}{r_k},\frac{\tau_k-t_k}{r_k^2}\right)+\frac{y_k}{r_k}e_{N+1}.
$$
Therefore, $S_k \to S$ as $k \to +\infty$, for some $S\in Q^\infty$, up to passing to a subsequence and so, as above, we have $|\partial_{ij}\Bar{w}({S})-\partial_{ij}\Bar{w}(\bar{S})|\ge C\delta_0$ which shows our claim.

\smallskip

\emph{Step 5. The equation of the limit $\bar{w}$.} In this step, we derive the equation of $\bar{w}$: as in the steps above, we divide the proof in two additional steps (\textbf{Case 1} and \textbf{Case 2}).

\smallskip

\noindent \textbf{Case 1:} In this case, we have $r_k/y_k\to 0$ as $k\to +\infty$ and $\hat{z}_k = z_k$. Further, if $\bar{A}^{(k)}(z,t) := A^{(k)}(r_kz+\hat{z}_k,r_k^2t+t_k)$, $\bar{\rho}_k(y) := r_ky+y_k$ and $(\Bar{z},\Bar{t}) := \lim_{k\to+\infty}(\Hat{z_k},t_k)$, then, by the regularity assumptions on $A^{(k)}$, one has $\bar{A}^{(k)}\to\bar{A}$ as $k \to +\infty$, where $\bar{A} := \lim_{k\to+\infty} A^{(k)}(\Bar{z},\bar{t})$ is a symmetric matrix with constant coefficients satisfying  \eqref{eq:UnifEll}.

We claim that $\Bar{w}$ is an entire solution to
\begin{equation}\label{eq:c2:claim:not:degenerate}
\partial_t \Bar{w}-\div(\bar{A}\nabla\Bar{w})=0\quad\text{in }\mathbb{R}^{N+2}.
\end{equation} 

Let us fix a compact set $ K\subset Q^\infty$. By \eqref{eq:c2:pointwisely} and using estimates in parabolic H\"older spaces, ${w}_k$ satisfies 
    \begin{align*}
   &\Big|\partial_t {w}_k - \sum_{i,j=1}^{N+1}(\bar{A}_{i,j}^{(k)} \partial_{ij} {w}_k)\Big| \\
    &=\frac{1}{{[u_k]_{C^{2,\alpha}_p(Q_1^+)} r_k^{\alpha}}}\Big|
\partial_t u_k(r_k z+z_k,r_k^2 t+t_k)
 -  \sum_{i,j=1}^{N+1}({A}_{i,j}^{(k)} \partial_{ij} {u}_k)(r_k z+z_k,r_k^2 t+t_k)\\
   &    \quad-\partial_t u_k({z}_k,t_k) +  \sum_{i,j=1}^{N+1}{A}_{i,j}^{(k)} (r_k z+z_k,r_k^2 t+t_k)\partial_{ij} {u}_k(z_k,t_k)
    \Big|\\
     &=\frac{1}{{[u_k]_{C^{2,\alpha}_p(Q_1^+)} r_k^{\alpha}}}\Big|
     g_k(r_k z+z_k,r_k^2 t+t_k)
     +\frac{a\big(\sum_{j=1}^{N+1}A^{(k)}_{N+1,j}\partial_j u_k+F^{(k)}_{N+1}\big)(r_k z+z_k,r_k^2 t+t_k)}{r_ky+y_k}\\
   &    \quad-\partial_t u_k({z}_k,t_k) +  
     \sum_{i,j=1}^{N+1}{A}_{i,j}^{(k)}\partial_{ij} {u}_k(z_k,t_k)
      + \sum_{i,j=1}^{N+1}\big({A}_{i,j}^{(k)} (r_k z+z_k,r_k^2 t+t_k)-A^{(k)}_{i,j}(z_k,t_k)\big)\partial_{ij} {u}_k(z_k,t_k) 
     \Big|\\
     &\le \frac{1}{{[u_k]_{C^{2,\alpha}_p(Q_1^+)} r_k^{\alpha}}}\Big|
     g_k(r_k z+z_k,r_k^2 t+t_k)
     +\frac{a\big(\sum_{j=1}^{N+1}A^{(k)}_{N+1,j}\partial_j u_k+F^{(k)}_{N+1}\big)(r_k z+z_k,r_k^2 t+t_k)}{r_ky+y_k}\\
     &\quad - g_k(z_k,t_k)
     -\frac{a\big(\sum_{j=1}^{N+1}A^{(k)}_{N+1,j}\partial_j u_k+F^{(k)}_{N+1}\big)(z_k,t_k)}{y_k}\Big|
     +C  \frac{\|D^2 u_k\|_{L^\infty(Q_{1}^+)}}{[u_k]_{C^{2,\alpha}_p(Q_1^+)}}\\
     &=  \frac{\big|g_k(r_k z+z_k,r_k^2 t+t_k)-g_k(z_k,t_k)\big|}{{[u_k]_{C^{2,\alpha}_p(Q_1^+)} r_k^{\alpha}}}+\frac{C\|D^2 u_k\|_{L^\infty(Q_1^+)}}{[u_k]_{C^{2,\alpha}_p(Q_1^+)}}\\
   & \quad+ \frac{a}{{[u_k]_{C^{2,\alpha}_p(Q_1^+)} r_k^{\alpha}}}\Big|  
   \frac{ H_k(r_k z+z_k,r_k^2 t+t_k) }{r_k y+y_k} -\frac{H_k(z_k, t_k)}{y_k}  
   \Big|=\text{I}+\text{II}+\text{III},
    \end{align*}
where $C > 0$ is a new constant independent of $k$ (here and below the constant $C > 0$ depends on $K$: we omit this dependence to simplify the exposition) and
$$ 
H_k(z,t) = \sum_{j=1}^{N+1}(A^{(k)}_{N+1,j}\partial_j u_k)(z,t)+F^{(k)}_{N+1}(z,t),
$$ 
which satisfies
$H_k(x_k,0,t_k) = 0$, $\D H_k(x_k,0,t_k) = \partial_y H_k(x_k,0,t_k) e_{N+1}$, $H_k(z,t)/y \in C^{0,\alpha}(Q_1^+)$ by Lemma \ref{lemma-2.3-TTV} and $[H_k/y]_{C^{0,\alpha}(Q_1^+)}\le C[\D H_k]_{C^{1,\alpha}(Q_1^+)} \le C [u_k]_{C^{2,\alpha}_p(Q_1^+)}$, by the assumption \eqref{eq:c2:XX:contradiction}.

Now, by \eqref{eq:c2:XX:contradiction} and \eqref{eq:c2:interpolation:g}, we can estimate I as follows
\[
\text{I} : =\frac{\big|g_k(r_k z+z_k,r_k^2 t+t_k)-g_k(z_k,t_k)\big|}{[u_k]_{C^{2,\alpha}_p(Q_1^+)} r_k^{\alpha}}\le \frac{[g_k]_{C^{0,\alpha}_p(Q_1^+)}}{[u_k]_{C^{2,\alpha}_p(Q_1^+)}}\le \frac{1}{k}\to0,
\]
as $k\to+\infty$. The term II vanishes as well as $k\to+\infty$, by similar considerations. Finally, let us prove that III vanishes as $k\to+\infty$. First,
 \begin{align*}
&\Big|  
   \frac{ H_k(r_k z+z_k,r_k^2 t+t_k) }{r_k y+y_k} -\frac{H_k(z_k, t_k)}{y_k}  
   \Big| = \Big|  
   \frac{ H_k(r_k z+z_k,r_k^2 t+t_k) }{r_k y+y_k} -\frac{H_k(z_k, t_k)}{r_ky+y_k} -   \frac{r_ky}{y_k}\frac{ H_k(z_k,t_k) }{r_k y+y_k}
   \Big|\\
   & \le  \Big|  
   \frac{ H_k(r_k z+z_k,r_k^2 t+t_k) }{r_k y+y_k} 
   -\frac{H_k(z_k, t_k)}{r_ky+y_k} - \frac{\D H_k(z_k,t_k) \cdot r_k z}{r_k y+ y_k}   \Big| \\
  & + \Big| \frac{\D H_k(z_k,t_k) \cdot r_k z}{r_k y+ y_k} 
   - \frac{r_ky}{y_k}\frac{ H_k(z_k,t_k) }{r_k y+y_k}
   \Big| = \text{III}_i + \text{III}_{ii}.
\end{align*}
 By using the parabolic first order expansion of $H_k$, \eqref{eq:c2:XX:contradiction} and $r_k y + y_k \ge y_k/2$, one has that
 \begin{equation}\label{eq:III_i}
 |\text{III}_i|\le C \frac{[H_k]_{C^{1,\alpha}_p(Q_1^+)} r_k^{1+\alpha}}{r_k y+y_k} \le C [u_k]_{C^{2,\alpha}_p(Q_1^+)}r_k^{1+\alpha}y_k^{-1}.
 \end{equation}
Instead, we estimate the term III$_{ii}$ in the following way
\begin{align}\label{eq:III_ii}
\begin{split}
& |\text{III}_{ii}| \le \Big| \frac{\D H_k(z_k,t_k) \cdot r_k z}{r_k y+ y_k} - \frac{\D H_k(x_k,0,t_k) \cdot r_k z}{r_k y+ y_k}\Big| + \Big|
   \frac{\D H_k(x_k,0,t_k) \cdot r_k z}{r_k y+ y_k} - \frac{r_ky}{y_k}\frac{ H_k(z_k,t_k) }{r_k y+y_k}
   \Big|\\
&  \le C [H_k]_{C^{1,\alpha}_p(Q_1^+)} r_k y_k^{ \alpha-1}
\le C [u_k]_{C^{2,\alpha}_p(Q_1^+)}r_k y_k^{\alpha -1},
\end{split}
\end{align}
where, in order to estimate the second term in the previous inequality we have used the properties of $H_k$ stated above. Hence, combining \eqref{eq:III_i} and \eqref{eq:III_ii} we have that 
\[
|\text{III}| \le C \frac{r_k}{y_k} +  C \Big(\frac{r_k}{y_k}\Big)^{1-\alpha} \to 0, \quad \text{as }k\to+\infty,
\]
since in \textbf{Case 1}, $r_k/y_k \to 0$.
\[
\partial_t {w}_k - \sum_{i,j=1}^{N+1}(\bar{A}_{i,j}^{(k)} \partial_{ij} {w}_k) \to \partial_t \bar{w}-\sum_{i,j=1}^{N+1}\bar{A}_{i,j}  \partial_{ij}\bar{w} \quad \text{ locally uniformly in } \R^{N+2},
\]
as $k\to+\infty$ and hence, passing to the limit as $k\to+\infty$ into the equation of $w_k$ above \eqref{eq:c2:claim:not:degenerate} follows.

\

\noindent \textbf{Case 2:} In this case, we have $\hat{z_k}=(x_k,0)$ and $r_k/y_k\le C$ for some $C > 0$ independent of $k$. We claim that $\Bar{w}$ is a entire solution to    
    \begin{equation}\label{eq:c2:claim:degenerate}
    \begin{cases}
       y^a\partial_t \Bar{w}-\div(y^a\bar{A}\nabla\Bar{w})=0 &{\rm in  }\hspace{0.1cm}\mathbb{R}^{N+1}_+\times \mathbb{R},\\
        \displaystyle{\lim_{y\to0^+}}y^a\bar{A}\nabla \Bar{w}\cdot e_{N+1}=0&{\rm on  }\hspace{0.1cm}\partial\R_+^{N+1}\times\R.
    \end{cases}
    \end{equation}
    Let us fix a compact set $ K\subset Q^\infty$.  By using \eqref{eq:c2:pointwisely}, \eqref{eq:c2:boundary:pointwisely} and the fact that $(\hat{z}_k,t_k)$ belongs to $\partial^0Q_1^+$, ${w}_k$ satisfies 
    \begin{align*}
    \mathcal{L}w_k &:= \partial_t {w}_k - \sum_{i,j=1}^{N+1}(\bar{A}_{i,j}^{(k)} \partial_{ij} {w}_k) - \frac{a}{y}\sum_{j=1}^{N+1}\big(\bar{A}^{(k)}_{N+1,j}\partial_j w_k\big) \\
    &=\frac{1}{{[u_k]_{C^{2,\alpha}_p(Q_1^+)} r_k^{\alpha}}}\Big[
    \partial_t u_k(r_k z+\hat{z}_k,r_k^2 t+t_k)
 -  \sum_{i,j=1}^{N+1}\big({A}_{i,j}^{(k)} \partial_{ij} {u}_k\big)(r_k z+\hat{z}_k,r_k^2 t+t_k)\\
 &\quad -\frac{a}{r_k y}  \sum_{j=1}^{N+1}\big({A}^{(k)}_{N+1,j}\partial_j u_k\big)(r_kz+\hat{z}_k,r_k^2t+t_k)
 -\partial_t u_k(\hat{z}_k,t_k)
 -  \sum_{i,j=1}^{N+1}{A}_{i,j}^{(k)}(r_kz+\hat{z}_k,r_k^2t+t_k) \partial_{ij} {u}_k(\hat{z}_k,t_k)\\
 &\quad +\frac{ a}{r_k y}\sum_{j=1}^{N+1}{A}^{(k)}_{N+1,j}(r_kz+\hat{z}_k,r_k^2t+t_k)\partial_j u_k(\hat{z}_k,t_k)
 +\frac{a}{ r_k y}\sum_{i,j=1}^{N+1}{A}^{(k)}_{N+1,j}(r_kz+\hat{z}_k,r_k^2t+t_k)\partial_{ij}u_k(\hat{z}_k,t_k)r_k x_i \Big]\\
 &=\frac{g_k(r_k z+z_k,r_k^2 t+t_k)-g_k(\hat{z}_k,t_k)-\sum_{i,j=1}^{N+1}\Big({A}_{i,j}^{(k)}(r_kz+\hat{z}_k,r_k^2t+t_k)-{A}_{i,j}^{(k)}(\hat{z}_k,t_k) \Big)\partial_{ij} {u}_k(\hat{z}_k,t_k)}{{[u_k]_{C^{2,\alpha}_p(Q_1^+)} r_k^{\alpha}}} \\
    &\quad+\frac{a}{{[u_k]_{C^{2,\alpha}_p(Q_1^+)} r_k^{\alpha}}} \bigg[\frac{F^{(k)}_{N+1}(r_kz+\hat{z}_k,r_k^2t+t_k)}{r_k y}
    -\partial_y\Big(  \sum_{j=1}^{N+1}A_{N+1,j}\partial_j u+F_{N+1}  \Big)(\hat{z}_k,t_k)   \\
 &\quad +\frac{ 1}{r_k y}\sum_{j=1}^{N+1}{A}^{(k)}_{N+1,j}(r_kz+\hat{z}_k,r_k^2t+t_k)\partial_j u_k(\hat{z}_k,t_k)
 +\frac{1}{r_k  y}\sum_{i,j=1}^{N+1}{A}^{(k)}_{N+1,j}(r_kz+\hat{z}_k,r_k^2t+t_k)\partial_{ij}u_k(\hat{z}_k,t_k)r_kx_i \bigg] \\
 &:= \text{J} + \text{JJ}.
\end{align*}
Similar to \textbf{Case 1}, \text{J} vanishes as $k\to+\infty$ (see the proof for I and II above). We are left to treat $\text{JJ}$. By Lemma \ref{lemma:derivative:i:solution} and Theorem \ref{teo-C^1,alpha}, we may differentiate $u_k$ w.r.t. $x_i$ ($i=1,\dots,N$) and $\partial_i u$ satisfies the following conormal boundary condition
    \begin{equation}\label{eq:c2:conormal:derivative}
    \lim_{y\to0^+}\partial_i\big(\sum_{j=1}^{N+1}A^{(k)}_{N+1,j}\partial_j u_k+F^{(k)}_{N+1}\big)=0,
    \end{equation}
and thus, recalling the conormal boundary condition of $u_k$, we deduce
$$
\text{JJJ} := -\frac{\sum_{j=1}^{N+1}\big(A^{(k)}_{N+1,j}\partial_j u_k\big)(\hat{z}_k,t_k)
    +F^{(k)}_{N+1}(\hat{z}_k,t_k)+
\sum_{i=1}^{N}\partial_i\big(\sum_{j=1}^{N+1}A^{(k)}_{N+1,j}\partial_j u_k+F^{(k)}_{N+1}\big)(\hat{z}_k,t_k)r_k x_i    
    }{r_k y} = 0.
$$
Adding \text{JJ} and \text{JJJ}, expanding $F^{(k)}$ and $A^{(k)}$ at order one and using the estimates in parabolic H\"older spaces, we obtain
    \begin{align*}
    &\Big|\frac{F^{(k)}_{N+1}(r_kz+\hat{z}_k,r_k^2t+t_k)}{r_k y}
    -\partial_y\Big(  \sum_{j=1}^{N+1}A_{N+1,j}\partial_j u+F_{N+1}  \Big)(\hat{z}_k,t_k)   \\
 &+\frac{ 1}{r_k y}\sum_{j=1}^{N+1}{A}^{(k)}_{N+1,j}(r_kz+\hat{z}_k,r_k^2t+t_k)\partial_j u_k(\hat{z}_k,t_k)
 +\frac{1}{ r_k y}\sum_{i,j=1}^{N+1}{A}^{(k)}_{N+1,j}(r_kz+\hat{z}_k,r_k^2t+t_k)\partial_{ij}u_k(\hat{z}_k,t_k)r_kx_i \Big|\\
 &= \Big|\frac{F^{(k)}_{N+1}(r_kz+\hat{z}_k,r_k^2t+t_k)-F^{(k)}_{N+1}(\hat{z}_k,t_k)-\sum_{i=1}^{N+1}\partial_i(F^{(k)}_{N+1})(\hat{z}_k,t_k)r_kx_i}{r_k y}\\
 &+\frac{\sum_{j=1}^{N+1}\Big({A}^{(k)}_{N+1,j}(r_kz+\hat{z}_k,r_k^2t+t_k)
- {A}^{(k)}_{N+1,j}(\hat{z}_k,t_k)-\sum_{i=1}^{N+1}\partial_i{A}^{(k)}_{N+1,j}(\hat{z}_k,t_k)r_k x_i
 \Big)\partial_{j}u_k(\hat{z}_k,t_k)}{r_k y}\\
 &+\frac{\sum_{i,j=1}^{N+1}\Big({A}^{(k)}_{N+1,j}(r_kz+\hat{z}_k,r_k^2t+t_k)
 -{A}^{(k)}_{N+1,j}(\hat{z}_k,t_k)\Big)\partial_{i,j}u_k(\hat{z}_k,t_k)r_k x_i
 }{r_k y}\Big|\\
 &\le Cr_k^\alpha\Big([F^{(k)}]_{C^{1,\alpha}_p(Q_1^+)}
 +[A^{(k)}]_{C^{1,\alpha}_p(Q_1^+)}\|\nabla u_k\|_{L^\infty(Q_1^+)}
 +[A^{(k)}]_{C^{0,1}_p(Q_1^+)}\|D^2 u_k\|_{L^\infty(Q_1^+)}
 \Big)\le \frac{Cr_k^\alpha[u_k]_{C^{2,\alpha}_p(Q_{1/2}^+)}}{k}.
\end{align*}
Consequently, $|\mathcal{L}w_k| = o(1)$, as $k\to+\infty$. As in \textbf{Case 1}, by \emph{Step 3}, one has 
    $$
\mathcal{L}w_k\to \partial_t \bar{w} - \sum_{i,j=1}^{N+1}\bar{A}_{i,j}\partial_{ij}\bar{w} - \frac{a}{y}\sum_{j=1}^{N+1}\bar{A}_{N+1,j}\partial_{j}\bar{w} \quad \text{ locally uniformly in } \R_+^{N+1},
$$
as $k\to+\infty$, and so $\bar{w}$ satisfies the equation in \eqref{eq:c2:claim:degenerate} in the classical sense. It remains to prove that $\bar{w}$ satisfies the conormal boundary condition in \eqref{eq:c2:claim:degenerate}. Since $u_k$ satisfies \eqref{eq:c2:conormal:derivative} and following the arguments above, we find
\begin{align*}
&\Big|\sum_{j=1}^{N+1}\big(\bar{A}^{(k)}_{N+1,j}\partial_j{w}_k\big)\Big| =  \frac{1}
{{[u_k]_{C^{2,\alpha}_p(Q_1^+)} r_k^{1+\alpha}}}\Big|\sum_{j=1}^{N+1}
\big({A}^{(k)}_{N+1,j}\partial_j{u}_k\big)(r_kz+\hat{z}_k,r_k^2t+t_k) \\
&\quad-\sum_{j=1}^{N+1}{A}^{(k)}_{N+1,j}(r_kz+\hat{z}_k,r_k^2t+t_k)\partial_j{u}_k(\hat{z}_k,t_k)  
-\sum_{j=1}^{N+1}{A}^{(k)}_{N+1,j}(r_kz+\hat{z}_k,r_k^2t+t_k)\partial_{ij}{u}_k(\hat{z}_k,t_k)  r_kx_i
 \Big|\\
&=\frac{1}
{{[u_k]_{C^{2,\alpha}_p(Q_1^+)} r_k^{1+\alpha}}}\Big|-F^{(k)}_{N+1}(r_kz+\hat{z}_k,r_k^2t+t_k)\\
&\quad-\sum_{j=1}^{N+1}\Big(   {A}^{(k)}_{N+1,j}(r_kz+\hat{z}_k,r_k^2t+t_k)-{A}^{(k)}_{N+1,j}(\hat{z}_k,t_k)  
    -\sum_{i=1}^{N+1}\partial_i({A}^{(k)}_{N+1,j})(\hat{z}_k,t_k)r_k x_i  
\Big)\partial_j{u}_k(\hat{z}_k,t_k)  \\
&\quad-\sum_{j=1}^{N+1}\big({A}^{(k)}_{N+1,j}\partial_j{u}_k\big)(\hat{z}_k,t_k)
-\sum_{i,j=1}^{N+1}\big(\partial_i{A}^{(k)}_{N+1,j}\partial_j{u}_k\big)(\hat{z}_k,t_k)r_kx_i-\sum_{i,j=1}^{N+1}\big({A}^{(k)}_{N+1,j}\partial_{ij}{u}_k\big)(\hat{z}_k,t_k)r_kx_i\\
&\quad -\sum_{i,j=1}^{N+1}\Big(   {A}^{(k)}_{N+1,j}(r_kz+\hat{z}_k,r_k^2t+t_k)-{A}^{(k)}_{N+1,j}(\hat{z}_k,t_k)  \Big)\partial_{ij}u_k(\hat{z}_k,t_k) r_k x_i\Big|\\
&\le\frac{1}
{{[u_k]_{C^{2,\alpha}_p(Q_1^+)} r_k^{1+\alpha}}}
\Big|-F^{(k)}_{N+1}(r_kz+\hat{z}_k,r_k^2t+t_k)+F^{(k)}_{N+1}(\hat{z}_k,t_k) + \sum_{i=1}^{N+1} \partial_i F^{(k)}_{N+1}(\hat{z}_k,t_k) r_k x_i\\
&\quad- \sum_{i=1}^{N+1} \partial_i F^{(k)}_{N+1}(\hat{z}_k,t_k) r_k x_i- \sum_{i,j=1}^{N+1} \partial_i \big({A}^{(k)}_{N+1,j}\partial_j{u}_k\big)(\hat{z}_k,t_k)r_kx_i\Big|+o(1)\\
&\le \frac{\Big|\partial_y\Big(\sum_{j=1}^{N+1}
{A}^{(k)}_{N+1,j}\partial_j{u}_k+F^{(k)}_{N+1}\Big)(\hat{z}_k,t_k)r_ky\Big|
}
{{[u_k]_{C^{2,\alpha}_p(Q_1^+)} r_k^{1+\alpha}}}+o(1)=o(1),
\end{align*}
as $k \to + \infty$. Thus, passing to the limit as $y\to0^+$, we obtain 
$$
\lim_{y \to 0^+}\Big|\sum_{j=1}^{N+1}\big(\bar{A}^{(k)}_{N+1,j}\partial_j{w}_k\big) \Big|\le o(1),
$$
and thus, taking the limit as $k\to+\infty$, it follows 
$$ 
\displaystyle{\lim_{y\to0^+}}\bar{A}\nabla \Bar{w}\cdot e_{N+1}=0.
$$
Combining this with the fact that $\bar{w} \in C^{2,\alpha}_p$ by \eqref{eq:c2:global:holder} and recalling that $a>-1$, we have
\[
\lim_{y\to0^+}y^a\bar{A}\nabla \Bar{w}\cdot e_{N+1 } = \lim_{y\to0^+}y^{1+a} \lim_{y\to0^+} \frac{\bar{A}\nabla \Bar{w}\cdot e_{N+1}}{y} = \partial_y(\bar{A}\nabla \Bar{w}\cdot e_{N+1})|_{y=0} \cdot \lim_{y\to0^+}y^{1+a} = 0,
\]
and so, the proof of \eqref{eq:c2:claim:degenerate} is completed.

\smallskip
    
    \emph{Step 6. Liouville theorems.} Since $\Bar{w} \in C^{2,\alpha}_p(Q^\infty)$, see \eqref{eq:c2:global:holder}, it satisfies the growth condition 
    $$
    |\bar{w}(z,t)|\le C(1+(|z|^2+|t|)^{2+\alpha})^{1/2}.
    $$
Moreover, $\Bar{w}$ has at least one non-constant second derivative and is an entire solution to \eqref{eq:c2:claim:not:degenerate} or \eqref{eq:c2:claim:degenerate}.
    Then, in \textbf{Case 1} we can invoke the Liouville Theorem for the heat equation (see \cite{AudFioVit24}*{Remark 5.3}) and in \textbf{Case 2} we can invoke the Liouville Theorem \ref{teo:polynomial:liouville} to reach the desired contradiction. 
\smallskip
    
    \emph{Step 7.} We complete the analysis, considering the case when 
    $$
    L_k=[\partial_i u]_{C^{0,\frac{1+\alpha}{2}}_t(Q_{1/2}^+)},
    $$ 
    for some $i\in\{1,\dots,N+1\}$. We give a short sketch, pointing out  the main differences respect to what did above. 

\smallskip

    We take two sequences of points $P_k=(z_k,t_k),\Bar{P}_k=(z_k,\tau_k)\in Q^+_{1/2}$, such that
\begin{equation}\label{eq:explosion:time:seminorm}
\frac{|\partial_i u_k(z_k,t_k)-\partial_i u_k(z_k,\tau_k)|}{|t_k-s_k|^\frac{1+\alpha}{2}}\ge\frac{L_k}{2},
\end{equation}
and set $r_k:=d_p(P_k,\bar{P}_k)=|t_k-\tau_k|^{1/2}$. We define the blow-up sequence $w_k$ as in \eqref{w_k degenerate C^2}, centered in $P_k$.

\smallskip

The \emph{Steps 3, 5, 6} are the same as above. The only crucial difference is in \emph{Step 4}: in this case, one has that $\partial_i \bar{w}$ is non-constant in $t$. Indeed,
\begin{align*}
    &\Big|\partial_i w_k\left(0,\frac{t_k-\tau_k}{r_k^2}\right)-\partial_i w_k(0,0)\Big| \ge \frac{L_k}{2[u_k]_{C^{2,\alpha}_p(Q_1^+)}}\ge C_N\delta_0.
\end{align*}
Taking the limit as $k\to+\infty$, we obtain that $|\partial_i\bar{w}(0,\bar{t})-\partial_i\bar{w}(0,0)|\ge C_N\delta_0$, where $\bar{t}=\lim_{k\to+\infty}\frac{t_k-\tau_k}{r_k^2}$. This allows as to conclude the proof of \eqref{eq:c2:XX:estimate} by applying Theorem \ref{teo:polynomial:liouville}.

\smallskip

    \emph{Step 8. Conclusion.} Finally, we briefly explain why \eqref{eq:c2:XX:estimate} implies \eqref{eq:c2} (see \cite{XX22}*{Theorem 2.20 and Lemma 2.27} in the elliptic setting). First, by using a covering argument and the interpolation inequalities in \eqref{lemma:interpolation}, we have that \eqref{eq:c2:XX:estimate} is satisfied in every $Q^+_\rho(P_0)\subset Q_1^+$, that is,
    \begin{equation}\label{eq:c2:XX:estimate2}
[ u]_{C^{2,\alpha}_p(Q^+_{\rho/2}(P_0))}\le \delta  [ u]_{C^{2,\alpha}_p(Q^+_\rho(P_0))}+C_\delta \left(\|u\|_{L^\infty(Q_1^+)}+\|f\|_{C^{0,\alpha}_p(Q_1^+)}+
    \|F\|_{C^{1,\alpha}_p(Q_1^+)}\right).
\end{equation}  
    Now, let us define the seminorm
    \begin{equation}\label{eq:def:holder*}
    [u]^*_{2,\alpha,Q_1^+}:=\sup_{Q^+_\rho(P_0)\subset Q_1^+}\rho^{2+\alpha}[u]_{C^{2,\alpha}_p(Q^+_{\rho/2}(P_0))}.
    \end{equation}
    By using the sub-additivity of the H\"older seminorms respect to unions of convex sets, one can prove that
    \begin{equation}\label{eq:conclusion1}
    [u]^*_{2,\alpha,Q_1^+}\le C \sup_{Q^+_\rho(P_0)\subset Q_1^+}\rho^{2+\alpha}[u]_{C^{2,\alpha}_p(Q^+_{\rho/4}(P_0))},
    \end{equation}
    for some constant $C>0$ depending only on $N$ and $\alpha$. Then, by \eqref{eq:c2:XX:estimate2} and \eqref{eq:def:holder*}, we obtain
    $$\rho^{2+\alpha}[ u]_{C^{2,\alpha}_p(Q^+_{\rho/4}(P_0))}\le \delta [u]^*_{2,\alpha,Q_1^+}+C_\delta \left(\|u\|_{L^\infty(Q_1^+)}+\|f\|_{C^{0,\alpha}_p(Q_1^+)}+
    \|F\|_{C^{1,\alpha}_p(Q_1^+)}\right).$$
    Taking the supremum over $Q^+_\rho(P_0)\subset Q_1^+$ and recalling \eqref{eq:conclusion1}, it follows
    $$\frac{1}{C} [u]^*_{2,\alpha,Q_1^+}\le \delta [u]^*_{2,\alpha,Q_1^+}+C_\delta \left(\|u\|_{L^\infty(Q_1^+)}+\|f\|_{C^{0,\alpha}_p(Q_1^+)}+
    \|F\|_{C^{1,\alpha}_p(Q_1^+)}\right).$$
    Hence our statement follows by taking $\delta>0$ small enough and using the interpolation inequality \eqref{lemma:interpolation}.
\end{proof}
\subsection{A regularization scheme} In this second step, we proceed with a regularization argument: this allows to apply the a priori estimates above and prove Theorem \ref{teo1-C2}.
\begin{lem}\label{lemma:approx:c2}
Let $N\ge1$, $a>-1$, $r\in(0,1)$, $\alpha\in (0,1)$. Let $A \in C^\infty(Q_1^+)$ satisfying \eqref{eq:UnifEll} and $f,F\in C^\infty(Q_1^+)$, and let $u$ be a weak solution to \eqref{eq:1}. Then $u\in C^{2,\alpha}_p(Q_{r}^+)$. 
\end{lem}
\begin{proof} We fix $0<r<r'<1$. 
For every $i=1,\dots,N$, by the regularity assumption on $A$, $f$ and $F$ and Lemma \ref{lemma:derivative:i:solution}, we have that $\partial_{x_i}u$ solves \eqref{eq:solution.i.derivative} in $Q_{r'}^+$ and, by Theorem \ref{teo-C^1,alpha}, we deduce that $\partial_{x_i}u\in C^{1,\alpha}_p(Q_r^+)$.
Analogously, by Lemma \ref{lemma:derivative:t:solution}, $\partial_{t}u$ solves \eqref{eq:solution.t.derivative} in $Q_{r'}^+$ and, by Theorem \ref{teo-C^1,alpha}, we deduce that $\partial_{t}u\in C^{1,\alpha}_p(Q_r^+)$. 
To conclude, we need to prove that $\partial_{y}u\in C^{1,\alpha}_p(Q_r^+)$.

Using the regularity of $\nabla u$ and $\partial_tu$ obtained above, we may rewrite the equation of $u$ as 
\begin{equation}\label{eq:TTV:y}
\partial_y\big(y^a (A\nabla u+F )\big)\cdot e_{N+1} = y^a \Big[ \partial_t u-f-\sum_{i=1}^{N}\partial_{x_i}((A\D u + F)\cdot e_i) \Big] := y^a g,
\end{equation}
in the weak sense, where $g\in C^{0,\alpha}_p(Q_r^+)$. Then, integrating in $y$ and using that $\lim_{y\to0^+}(A\nabla u+F)\cdot e_{N+1}=0$ (see Theorem \ref{teo-C^1,alpha}), one has
\begin{equation}\label{eq:TTV:2}
\psi(x,y,t):=(A\nabla u+F)\cdot{e_{N+1}}(x,y,t)=\frac{1}{y^a}\int_0^y s^ag(x,s,t)ds.
\end{equation}
Since $\partial_{x_i} u, \partial_t u \in C^{1,\alpha}_p(Q_r^+)$, we have $\partial_{x_i}\psi \in C^{0,\alpha}_p(Q_r^+)$ by definition, for every $i=1,\dots,N$, and $\partial_{t} \psi \in C^{0,\alpha}_p(Q_r^+)$. Consequently, $\psi\in C^{0,\frac{1+\alpha}{2}}_t(Q_r^+)$. Now, since $g\in C^{0,\alpha}_p(Q_r^+)$, Lemma \ref{lemma-2.4-TTV} yields $\partial_y \psi \in C^{0,\alpha}_p(Q_r^+)$ and thus $\psi \in C^{1,\alpha}_p(Q_r^+)$. Noticing that, by \eqref{eq:UnifEll}, we have
\begin{equation}\label{eq:ReopyU}
\partial_y u = \frac{\psi-\sum_{j=1}^N A_{N+1,j}\partial_j u-F_{N+1}}{A_{N+1,N+1}},
\end{equation}
it follows $\partial_y u \in C^{1,\alpha}_p(Q_r^+)$ and thus $u \in C^{2,\alpha}_p(Q_r^+)$.
\end{proof}
We are now ready to show Theorem \ref{teo1-C2}.
\begin{proof}[Proof of Theorem \ref{teo1-C2}]
Let us fix $0<r<R<1$ and let $u$ be a weak solution to \eqref{eq:1}. Let us consider a smooth cut-off function $\xi\in C_c^\infty(B_R)$, such that $0\le\xi\le 1$ and $\xi = 1$ in $B_{r}.$ Then, $v:=\xi u$ is a weak solution to
\begin{equation}\label{eq:xi:u:app}
\begin{cases}
y^a\partial_t v-\div (y^a A\nabla v)=y^a(\xi{f}-F\cdot\nabla\xi-A\nabla u\cdot\nabla\xi)+\div (y^a(\xi {F}-uA\nabla\xi)),&\text{ in }Q_R^+\\
\displaystyle{\lim_{y\to0^+}}y^a(A\nabla v+\xi {F}-uA\nabla\xi)\cdot e_{N+1}=0 & \text{ on }\partial^0Q_R^+\\
v=0 & \text{ on }\partial B_R^+\times I_R\\
v=\eta u & \text { on } B_R^+\times\{-R^2\}.
\end{cases}
\end{equation}
Let us denote with $\bar{A}$, $\bar{f}$ and $\bar{F}$ the even extensions of $A$, $f$ and $F$ w.r.t. $y$, respectively and let $A_\e:=\tilde{A}*\rho_\e$, $f_\e:=\tilde{f}*\rho_\e$ and $F_\e:=\tilde{F}*\rho_\e$, where $\{\rho_\e\}_{\e > 0}$ is a family of smooth mollifiers. Then, up to choose $\e$ small enough, $A_\e, f_\e, F_\e \in C_c^\infty(Q^+_R)$ and $A_\e$ satiafies \eqref{eq:UnifEll}. For every $\e\in(0,1)$, let $v_\e$ be the weak solution to 
\[
\begin{cases}
y^a\partial_t v_\e-\div (y^a A_\e\nabla v_\e)=y^a(\xi{f_\e}-F_\e \cdot\nabla\xi-A_\e\nabla u\cdot\nabla\xi)+\div (y^a(\xi {F}_\e - uA_\e\nabla\xi)),&\text{ in }Q_R^+\\
\displaystyle{\lim_{y\to0^+}}y^a(A_\e\nabla v_\e+\xi {F}_\e-uA_\e\nabla\xi)\cdot e_{N+1}=0 & \text{ on }\partial^0Q_R^+\\
v_\e=0 & \text{ on }\partial B_R^+\times I_R \\
v_\e= v & \text { on } B_R^+\times\{-R^2\}.
\end{cases}
\]
By the same compactness argument of Lemma \ref{lemma:derivative:t:solution} (or, equivalently, \cite{AudFioVit24}*{Lemma 4.3, Remark 4.4}), and by the classical theory of the Cauchy-Dirichlet problem in abstract Hilbert spaces, see \cite{lions}, we have that $v_\e\to v$ in $L^2(Q_R^+,y^a)$, which implies that $v_\e\to u$ in $L^2(Q_{r}^+,y^a)$ by the definition of $v$. On the other hand, since $\xi\equiv1$ in $B_r$, one has that $v_\e$ is a weak solution to 
\[
\begin{cases}
y^a\partial_t v_\e-\div (y^a D_\e\nabla v_\e)=y^a f_\e+\div (y^a {F}_\e) &\text{ in }Q_{r}^+,\\
\displaystyle{\lim_{y\to0^+}}y^a(D_\e\nabla v_\e+ {F}_\e)\cdot e_{N+1}=0 & \text{ on }\partial^0Q_{r}^+.
\end{cases}
\]
So, up to rescaling, Lemma \ref{lemma:approx:c2} yields that $v_\e \in C^{2,\alpha}_p(Q_{1}^+)$. 

On the other hand, by Proposition \ref{teo-C^2,alpha}, we deduce that $v_\e$ satisfies the desired estimate \eqref{eq:c2} in $Q_r^+$, uniformly in $\e > 0$. By the Arzel\`a-Ascoli theorem, we may thus take the limit as $\e\to0^+$ and complete the proof of \eqref{eq:c2}.
\end{proof}

%
%
\section{$C^{k+2,\alpha}_p$ regularity}\label{section:c^k+2}

In this section, we prove Theorem \ref{teo1} for any $k\geq1$ by combining some a priori estimates and an approximation argument. As anticipated in the introduction, we first deal with the case of a zero forcing term in the equation \eqref{eq:1}, i.e. $f= 0$. In this case, the main result follows by a simple iteration of the $C^{1,\alpha}_p$ and $C^{2,\alpha}_p$ estimates on partial derivatives. Secondly, we treat forcing terms $f\in C^{k,\alpha}_p$. In this case, the strategy is more involved and requires some additional and delicate steps (see Lemma \ref{approximation C^3}).

\subsection{Higher order Schauder estimates when $f=0$} We begin by treating the simpler case $f= 0$. 
\begin{proof}[Proof of Theorem \ref{teo1} when $f=0$.]
We proceed by induction. The initial step $k=0$ follows by Theorem \ref{teo1-C2}.

Let us fix $0<r<r'<1$ and assume that $A,F\in C^{j+2,\alpha}_p(Q_{r}^+)$ imply that \eqref{eq:ck:1} holds for $j=0,\dots,k$ and prove it for $k+1$. By Lemma \ref{lemma:derivative:i:solution} and the induction step we may differentiate the equation of $u$ w.r.t. $x_i$ to obtain $\partial_{x_i} u\in C^{k+2,\alpha}_p(Q_r^+)$ for every $i=1,\dots,N$ and 
\begin{align}\label{eq:f=0:i}
        \begin{aligned}
            \| \partial_{x_i}u\|_{C^{k+2,\alpha}_p(Q_{r}^+)}\le C\big(
    \|\partial_{x_i} u\|_{L^2(Q_{r'}^+,y^a)}+
    \|\partial_i F\|_{C^{k+1,\alpha}_p(Q_1^+)}
    \big)
    \le C\big(\|u\|_{L^2(Q_1^+,y^a)}+
    \|F\|_{C^{k+2,\alpha}_p(Q_1^+)}\big),
        \end{aligned}
    \end{align}
    for some $C>0$ which depends on $N$, $a$, $\lambda$, $\Lambda$, $r$, $\alpha$ and $\|A\|_{C^{k+2,\alpha}_p(Q_1^+)}$. 
On the other hand, By Lemma \ref{lemma:derivative:t:solution} and the induction step (noticing that in the case $k=0$ we use Theorem \ref{teo-C^1,alpha}) we may differentiate the equation of $u$ w.r.t. $t$ to obtain $\partial_{t} u\in C^{k+1,\alpha}_p(Q_r^+)$ and 
\begin{align}\label{eq:f=0:t}
        \begin{aligned}
            \|\partial_{t}u\|_{C^{k+1,\alpha}_p(Q_{r}^+)}\le C\big(
    \|\partial_{t}u\|_{L^2(Q_{r'}^+,y^a)}+
    \|\partial_t F\|_{C^{k,\alpha}_p(Q_1^+)}
    \big)
    \le C\big(\|u\|_{L^2(Q_1^+,y^a)}+
    \|F\|_{C^{k+2,\alpha}_p(Q_1^+)}\big),
        \end{aligned}
    \end{align}
    $N$, $a$, $\lambda$, $\Lambda$, $r$, $\alpha$ and $\|A\|_{C^{k+2,\alpha}_p(Q_1^+)}$.
 Repeating exactly the same argument of Lemma \ref{lemma:approx:c2} we obtain that the function $g$ defined in \eqref{eq:TTV:y} belongs to $C^{k+1,\alpha}_p(Q_r^+)$ and thus $\partial_y u \in C^{k+2,\alpha}_p(Q_r^+)$ which, in turn, implies $u\in C^{k+3,\alpha}_p(Q_r^+)$. Moreover, by using \eqref{eq:ReopyU}, \eqref{eq:f=0:i}, \eqref{eq:f=0:t} one has
 \begin{align}\label{eq:f=0:y}
        \begin{aligned}
            \|\partial_{y}u\|_{C^{k+2,\alpha}_p(Q_{r}^+)}
    \le C\big(\|u\|_{L^2(Q_1^+,y^a)}+
    \|F\|_{C^{k+2,\alpha}_p(Q_1^+)}\big),
        \end{aligned}
    \end{align}
    $N$, $a$, $\lambda$, $\Lambda$, $r$, $\alpha$ and $\|A\|_{C^{k+2,\alpha}_p(Q_1^+)}$.
Then, combining \eqref{eq:f=0:i}, \eqref{eq:f=0:t} and \eqref{eq:f=0:y} our statement follows.
\end{proof}
%
%
\subsection{Higher order Schauder estimates} Now, we consider the case $f\in C^{k,\alpha}_p$. As remarked in the introduction, when $k=1$, we can not use the same argument of the case $f=0$, since the function $\partial_t f$ is not well defined. In order to overcome this problem, we prove a priori $C^{3,\alpha}_p$-estimates and combining these with Lemma \ref{lemma:infinity-regularity} and Lemma \ref{approximation C^3}, we obtain our statement in the case $k=1$. For the general case $k\geq2$, one could possibly iterate the estimates obtained to prove the main result, as done in the case $f=0$. However, in order to keep the presentation uniform, we choose to iterate the full procedure (a priori estimates plus approximation) at any step.
\begin{pro}\label{teo-C^k,alpha}
      Let $N\ge1$, $a>-1$, $\alpha \in (0,1)$, $r\in(0,1)$ and $k \in\N$. Let $A\in C^{k+1,\alpha}_p(Q_1^+)$ satisfying \eqref{eq:UnifEll}, $f\in C^{k,\alpha}_p(Q_1^+)$ and $F\in C^{k+1,\alpha}_p(Q_1^+)$ and let $u\in C^{k+2,\alpha}_p(Q_1^+)$ be a weak solution to \eqref{eq:1}. 
      Then, there exists $C>0$, depending on $N$, $a$, $\lambda$, $\Lambda$, $r$, $\alpha$ and $\|A\|_{C^{k+1,\alpha}_p(Q_1^+)}$ such that
\begin{equation}\label{estimate C^k,alpha}
    \|u\|_{C^{k+2,\alpha}_p(Q_r^+)}\le C\left(
    \|u\|_{L^2(Q_1^+,y^a)}+
    \|f\|_{C^{k,\alpha}_p(Q_1^+)}+
    \|F\|_{C^{k+1,\alpha}_p(Q_1^+)}
    \right).
\end{equation}
\end{pro}
The proof of Proposition \ref{teo-C^k,alpha} crucially uses Lemma \ref{approximation C^3} below. In turn, in the proof of Lemma \ref{approximation C^3}, we exploit an approximation argument which relies on the following auxiliary result.
\begin{lem}\label{lemma:infinity-regularity}
Let $N\ge1$, $a>-1$, $r\in(0,1)$, $k\in\N$. Let $A \in C^{k+2,\alpha}_p(Q_1^+)$ satisfying \eqref{eq:UnifEll}, $f\in C^\infty(Q_1^+)$, $F\in C^{k+2,\alpha}_p(Q_1^+)$, and let $u$ be a weak solution to \eqref{eq:1}. Then $u\in C^{k+3,\alpha}_p(Q_{r}^+)$. 
\end{lem}
\begin{proof} It is enough to slightly modify the arguments of the proof of Theorem \ref{teo1} in the case $f=0$.
\end{proof}
\begin{lem}\label{approximation C^3}
    Let $N\ge1$, $a>-1$, $\alpha\in(0,1)$ and $k\in\N$. Let $D\in  C^{k+2,\alpha}_p(Q_1^+)$ be a diagonal matrix satisfying \eqref{eq:UnifEll}, $f\in C^{k+1,\alpha}_p(Q_1^+)$ and $F\in C^{k+2,\alpha}_p(Q_1^+)$. Let $\mu:=D_{N+1,N+1}$ and $g:=F_{N+1}$. Let $u\in C^{k+3,\alpha}_p(Q_1^+)$ be a weak solution to
    \begin{equation}\label{eq:u:C3}
    \begin{cases}
        y^{a}\partial_t u-{\div }(y^{a} D\nabla u)=y^{a}f+{\div}(y^{a} F) &{\rm in  }\hspace{0.1cm}Q_1^+,\\
       \displaystyle{\lim_{y\to0^+}}y^a(\mu\partial_y u+g)=0&{\rm on  }\hspace{0.1cm}\partial^0Q_1^+.
    \end{cases}
    \end{equation}
Then, the function $$w:=y^{-a}\partial_y\Big(y^a \Big(\partial_y u+\frac{g}{\mu}\Big)\Big) \in C^{k+1,\alpha}_p(Q_1^+),$$
is a weak solution to
\begin{equation}\label{eq:yy:solution}
\begin{cases}
        y^a\partial_t w-{\div }(y^{a} D\nabla w)={\div}(y^{a} \Tilde{F}) &{\rm in  }\hspace{0.1cm}Q_r^+,\\
        \displaystyle{\lim_{y\to0^+}}y^a\big(D\nabla w+\Tilde{F}\big)\cdot e_{N+1}=0&{\rm on  }\hspace{0.1cm}\partial^0Q_r^+,
    \end{cases}
\end{equation}
    where 
    \begin{align}\label{eq:TildeFDef}
        \begin{aligned}
        \Tilde{F}:=\partial_y D\nabla\Big(\partial_y u+\frac{g}{\mu}
    \Big)+\Big[\tilde{f}+\partial_y\mu\frac{a}{y}\Big(\partial_yu+\frac{g}{\mu}
    \Big)\Big]e_{N+1},
        \end{aligned}
    \end{align}
        and
        \begin{equation}\label{eq:HATF:C3-}
        \tilde{f}:=\partial_y f+\partial_y\div g+\div(\partial_y D \nabla u)+\partial_t\left(\frac{g}{\mu}\right)-\div\Big(D\nabla\Big(\frac{g}{\mu}\Big)\Big).
        \end{equation}
Moreover, 
\begin{equation}\label{eq:F:tilde:estimate}
\|\Tilde{F}\|_{C^{k,\alpha}_p(Q_r^+)}\le C\big(
\|f\|_{C^{k+1,\alpha}_p(Q_1^+)}+\|F\|_{C^{k+2,\alpha}_p(Q_1^+)}+\|u\|_{C^{k+2,\alpha}_p(Q_r^+)}
\big),
\end{equation}
for some $C>0$ depending only on $N$, $a$, $\lambda$, $\Lambda$, $r$, $\alpha$ and $\|A\|_{C^{k+2,\alpha}_p(Q_1^+)}$.
\end{lem}

\begin{proof}

\emph{Step 1.} First, we prove that
$$
w := y^{-a}\partial_y\Big(y^a \Big(\partial_y u+\frac{g}{\mu}\Big)\Big) = \partial_y \Big(\partial_y u+\frac{g}{\mu}\Big)+\frac{a}{y}\Big(\partial_y u+\frac{g}{\mu}\Big)\in C^{k+1,\alpha}_p(Q_1^+).
$$
By \eqref{eq:UnifEll}, we have $\mu\ge \lambda>0$ and so $\partial_y u+\frac{g}{\mu}\in C^{k+2,\alpha}_p(Q_1^+)$, thanks to the regularity assumptions on $\mu, g$ and $\mu$. By Theorem \ref{teo-C^1,alpha}, $u$ satisfies the conormal boundary condition 
    \begin{equation}\label{eq:Conorvck}
    \lim_{y\to0^+}\mu\partial_y u+g=0,
        \end{equation}
    and hence, by Lemma \ref{lemma-2.3-TTV}, we deduce that $\frac{a}{y} (\partial_y u+\frac{g}{\mu})\in C^{k+1,\alpha}_p(Q_1^+)$, which implies that $w\in C^{k+1,\alpha}_p(Q_1^+)$ by definition of $w$.
     
By similar considerations, it follows that $\tilde{f}\in C^{k,\alpha}_p(Q_1^+)$, where $\tilde{f}$ is defined in \eqref{eq:HATF:C3-}. Consequently, $\Tilde{F}\in C^{k,\alpha}_p(Q_1^+)$ (defined in \eqref{eq:TildeFDef}) and \eqref{eq:F:tilde:estimate} directly follows by definition.
   
\smallskip

\emph{Step 2.}
From this point we distinguish two cases as follows. If $k=0$, we assume that $D, f, F \in C^\infty(Q_1^+)$ and thus, by Lemma \ref{lemma:infinity-regularity}, $u \in C^\infty(Q_1^+)$ as well. We will recover our statement under the weaker assumptions $D\in  C^{2,\alpha}_p(Q_1^+)$, $f\in C^{1,\alpha}_p(Q_1^+)$ and $F\in C^{2,\alpha}_p(Q_1^+)$ throughout an approximation argument (see \emph{Step 3}). If $k\ge1$ we such approximation argument is not needed (this is because, when $k \geq 1$, the equation of $w$ is satisfied in the classical sense).

  We may rewrite \eqref{eq:u:C3} as
    \begin{equation}\label{eq.diag.approx}
    \partial_t u -\div(D\nabla u)-\frac{a}{y}(\mu\partial_y u+g)=f+\div F \quad\text{ in }Q_1^+.
    \end{equation}
    Differentiating the above equation w.r.t. $y$, we obtain
  \begin{equation}\label{eq.diag.approx.2}
      \partial_t (\partial_y u)-{\div}(D\nabla( \partial_{y}u))-{\div}(\partial_y D\nabla u)-\partial_y\left( \frac{a}{y}(\mu \partial_y u+g)\right)=\partial_y f+\partial_y \div F \quad\text{ in } Q_1^+.
\end{equation}
 Taking in account \eqref{eq.diag.approx.2} and setting $v:=y^a\left(\partial_y u+\frac{g}{\mu}\right)$, we obtain the equation of $v$
        \begin{align}\label{eq:hatf:c3}
        \begin{aligned}
            y^{-a}\partial_{t}v-\div(y^{-a}D\nabla v)&=\partial_t (\partial_{y}u)+\partial_t\left(\frac{g}{\mu}\right)-\div\left(D\nabla\left(\partial_y u+\frac{g}{\mu}\right)\right)-\partial_y\left(\frac{a}{y}(\mu\partial_y u+g)\right)\\
            &=\partial_t (\partial_{y}u)+\partial_t\left(\frac{g}{\mu}\right)-\div(D\nabla (\partial_{y}u))-\div\left(D\nabla\left(\frac{g}{\mu}\right)\right)-\partial_y\left(\frac{a}{y}(\mu\partial_y u+g)\right)\\
            &=\partial_y f+\partial_y\div\Bar{F}+\div(\partial_y D \nabla u)+\partial_t\left(\frac{g}{\mu}\right)-\div\left(D\nabla\left(\frac{g}{\mu}\right)\right):=\tilde{f} \quad \text{ in } Q_1^+,
            \end{aligned}
        \end{align}
 and thus, recalling that $\mu \geq \lambda > 0$ and \eqref{eq:Conorvck}, $v$ satisfies
\begin{equation}\label{eq.diag.approx.3}
    \begin{cases}
        y^{-a}\partial_t v-\div(y^{-a}D\nabla v)=y^{-a}(y^a\tilde{f})&\text{in }Q_1^+,\\
        v=0 &\text{on }\partial^0 Q_1^+.
    \end{cases}
\end{equation}
Differentiating \eqref{eq.diag.approx.3} w.r.t. $y$, we get
\begin{equation}\label{eq.diag.approx.4}
    \partial_t \partial_y v-{\div}(D\nabla \partial_y v)-{\div}(\partial_y D\nabla v)-\partial_y\left( \frac{a}{y}(\mu \partial_y v)\right)=\partial_y (y^a \tilde{f}) \quad\text{in }Q_1^+.
\end{equation}
Consequently,  $w = y^{-a}\partial_y v$ and satisfies pointwisely in $Q_1^+$
\begin{align*}
    y^a\partial_t w -\div(y^a D\nabla w)&= \partial_t \partial_y v-\div(D\nabla \partial_y v)+\left( \frac{a}{y}(\mu \partial_y v)\right)\\
    &=\partial_y(y^a\tilde{f})+\div(\partial_y D\nabla v)\\
    &=\partial_y(y^a\tilde{f})+\div\left(y^a\partial_y D \nabla\left(\partial_y u+\frac{g}{\mu}
    \right)\right)+\partial_y \left( y^a \partial_y\mu \frac{a}{y}\left( \partial_yu+\frac{g}{\mu}\right)\right).
\end{align*}
We need to establish that $w$ satisfies the boundary condition in \eqref{eq:yy:solution}. By the regularity assumptions and the fact that $v= 0$ on $\{y=0\}$, we can take the limit as $y\to0^+$ in the equation \eqref{eq.diag.approx.3} to get
\[
\lim_{y\to0^+}\Big[\mu \partial_{yy}v-\frac{a}{y}\mu \partial_y v+\partial_y \mu \partial_y v+y^a\tilde{f}
    \Big]=\lim_{y\to0^+}\Big[\partial_t v-\sum_{i=1}^{N}\partial_{x_i}(D_{i,i}\partial_{x_i}v)    \Big]=0,
    \]   
which turns out to be the boundary condition
\begin{align*}
    0=\lim_{y\to0^+}\left[y^a(\mu\partial_y w+\tilde{f})+\partial_y\mu \partial_yv\right] =\lim_{y\to0^+}y^a\left[\mu\partial_y w+\tilde{f}+\partial_y\mu\partial_y\left(\partial_y u+ \frac{g}{\mu}
    \right)+\partial_y \mu\frac{a}{y}\left(\partial_yu+\frac{g}{\mu}
    \right)\right].
\end{align*}
Hence, defining $\Tilde{F}$ as in \eqref{eq:TildeFDef}, it follows that $w$ is solution to \eqref{eq:yy:solution} as claimed.

\smallskip

\emph{Step 3.}
In this final step, we present the approximation argument which allows to complete the proof when $k=0$. First, by Theorem \ref{teo1-C2}, we have that 
\begin{align}\label{eq:estimates:for:compactness:c3}
\begin{aligned}
\|\Tilde{F}\|_{C^{0,\alpha}_p(Q_r^+)} &\le C\big(
\|f\|_{C^{1,\alpha}_p(Q_1^+)}+\|F\|_{C^{2,\alpha}_p(Q_1^+)}+\|u\|_{C^{2,\alpha}_p(Q_r^+)}\big) \\
&\le C\big(
\|f\|_{C^{1,\alpha}_p(Q_1^+)}+\|F\|_{C^{2,\alpha}_p(Q_1^+)}+\|u\|_{L^2(Q_1^+,y^a)}\big),
\end{aligned}
\end{align}
for some $C>0$ depending only on $N$, $a$, $\lambda$, $\Lambda$, $r$, $\alpha$ and $\|D\|_{C^{2,\alpha}_p(Q_1^+)}$.

The proof follows the approximation scheme done in the proof of Theorem \ref{teo1-C2}: it is enough to replace the matrix $A$ with the matrix $D$. Indeed, after regularizing the data (which we call $f_\e,F_\e,A_\e \in C^\infty(Q_1^+)$), and using Lemma \ref{lemma:infinity-regularity}, we can find a family of smooth solutions $v_\e \in C_c^\infty(Q_r^+)$ to
\[
\begin{cases}
y^a\partial_t v_\e-\div (y^a D_\e\nabla v_\e)=y^a f_\e+\div (y^a {F}_\e) &\text{ in }Q_{r}^+,\\
\displaystyle{\lim_{y\to0^+}}y^a(D_\e\nabla v_\e+ {F}_\e)\cdot e_{N+1}=0 & \text{ on }\partial^0Q_{r}^+.
\end{cases}
\]
 which converges to the original solution $u$ as $\e\to0^+$. 
Consequently, \emph{Step 2} yields that
$$
w_\e:=y^{-a}\partial_y\Big(y^a \Big(\partial_y v_\e+\frac{\Bar{f_\e}}{\mu_\e}\Big)\Big)
$$
is a solution to \eqref{eq:yy:solution} (with $D$ and $\tilde{F}$ replaced by $D_\e$ and $\tilde{F}_\e$) and $\tilde{F}_\e$, defined accordingly to \eqref{eq:TildeFDef}, satisfies  \eqref{eq:estimates:for:compactness:c3}. By Proposition \ref{teo-C^2,alpha} and the Arzel\'a-Ascoli theorem, one has that $v_\e \to u$ in $C^{2,\alpha}_p(Q_r^+)$, which implies that $w_\e \to w$ in $C^{0,\alpha}_p(Q_r^+)$. Then a slight modification of the argument in \cite{AudFioVit24}*{Lemma 4.2} shows that $w_\e$ converges to $w$ in the energy spaces and that $w$ is a weak solution to \eqref{eq:yy:solution}, as claimed.
\end{proof}
\begin{proof}[Proof of Proposition \ref{teo-C^k,alpha}]

    Let $\partial_i:=\partial_{x_i}$ for $i=1,\dots,N$. We proceed with an induction argument. The step $k=0$ has been proved in Proposition \ref{teo-C^2,alpha}. Let us assume that \eqref{estimate C^k,alpha} holds for $j=1,\dots,k\in\mathbb{N}$ and let us prove that it is valid for $k+1$. So, let $u\in C^{k+3,\alpha}_p(Q_1^+)$, $A,F\in C^{k+2,\alpha}_p(Q_1^+)$ and $f\in C^{k+1,\alpha}_p(Q_1^+)$.
    
Let us fix $0<r<r'<1$. First, for every $i=1,\dots,N$, by Lemma \ref{lemma:derivative:i:solution}, one has that $u_i:=\partial_{i}u$ solves \eqref{eq:solution.i.derivative} in $Q_{r'}^+$. Noticing that $u_i\in C^{k+2,\alpha}_p(Q_1^+)$, $A\in C^{k+2,\alpha}_p(Q_1^+)$, $\partial_i f\in C^{k,\alpha}_p(Q_1^+)$ and $\partial_i F, \partial_{i}A\nabla u\in C^{k+1,\alpha}_p(Q_1^+)$, we can use the inductive step to obtain 
    \begin{align}\label{eq:C^3:i}
        \begin{aligned}
            \|u_i\|_{C^{k+2,\alpha}_p(Q_{r}^+)}\le& C\left(
    \|u_i\|_{L^2(Q_{r'}^+,y^a)}+
    \|\partial_i f\|_{C^{k,\alpha}_p(Q_1^+)}+
    \|\partial_i F\|_{C^{k+1,\alpha}_p(Q_1^+)}
    \right)\\
    \le& C\left(\|u\|_{L^2(Q_1^+,y^a)}+
    \|f\|_{C^{k+1,\alpha}_p(Q_1^+)}+
    \|F\|_{C^{k+2,\alpha}_p(Q_1^+)}\right),
        \end{aligned}
    \end{align}
    where $C>0$ depends only on $N$, $a$, $\alpha$, $\lambda$, $\Lambda$, $\|A\|_{C^{k+2,\alpha}_p(Q_1^+)}$. It remains to prove that
    \begin{align*}
        &[u_{yyy}]_{C^{k,\alpha}_p(Q_r^+)}+[u_{yy}]_{C^{k,\frac{1+\alpha}{2}}_t(Q_r^+)}+[u_{ty}]_{C^{k,\alpha}_p(Q_r^+)}+[u_{t}]_{C^{k,\frac{1+\alpha}{2}}_t(Q_r^+)}\\
        &\le C\left(\|u\|_{L^2(Q_1^+,y^a)}+
    \|f\|_{C^{k+1,\alpha}_p(Q_1^+)}+
    \|F\|_{C^{k+2,\alpha}_p(Q_1^+)}\right).
    \end{align*}
    Let $D:=\text{diag}(A)$. It is immediate to check that $u$ solves
$$
\begin{cases}
        y^a\partial_t u-{\div}(y^a D\nabla u)=y^a\bar{f}+{\div}(y^a \bar{F})&\text{ in }Q_{r}^+\\
\displaystyle{\lim_{y\to0^+}}y^a(\mu\partial_y u+g)=0,&\text{ on }\partial^0Q_{r}^+
\end{cases}
$$
    where $\bar{F}:=((A-D)\nabla u\cdot e_{N+1})e_{N+1}+F$ and $\bar{f}:=f+\sum_{i,j=1}^{N+1}\partial_i ((A-D)_{i,j}\partial_j u)$, $g=\bar{F}_{N+1}$ and $\mu=A_{N+1,N+1}$. Furthermore, by \eqref{eq:C^3:i} and the definition of $\bar{F}$ and $\bar{f}$, we have that $\bar{F}\in C^{k+2,\alpha}_p(Q_{r'}^+)$, $\bar{f}\in C^{k+1,\alpha}_p(Q_{r'}^+)$ and 
    \begin{equation}\label{eq:bar:f:est}
    \|\bar{F}\|_{C^{k+2,\alpha}_p(Q_{r'}^+)}+\|\bar{f}\|_{C^{k+1,\alpha}_p(Q_{r'}^+)}\le C \big(
\|u\|_{L^2(Q_1^+,y^a)}+
    \|f\|_{C^{k+1,\alpha}_p(Q_1^+)}+
    \|F\|_{C^{k+2,\alpha}_p(Q_1^+)}    
    \big),
    \end{equation}
for some $C>0$ which depending only on $N$, $a$, $\alpha$, $\lambda$, $\Lambda$, $\|A\|_{C^{k+2,\alpha}_p(Q_1^+)}$.

By Lemma \ref{approximation C^3}, the function $w:=y^{-a}\partial_y\big(y^a( \partial_y u+g/{\mu})\big)$ belongs to $C^{k+1,\alpha}_p(Q_{r'}^+)$ and is a weak solution to 
\begin{equation*}
\begin{cases}
        y^a\partial_t w-{\div }(y^{a} D\nabla w)={\div}(y^{a} \Tilde{F}) &{\rm in  }\hspace{0.1cm}Q_{r'}^+,\\
        \displaystyle{\lim_{y\to0^+}}y^a\big(D\nabla w+\Tilde{F}\big)\cdot e_{N+1}=0&{\rm on  }\hspace{0.1cm}\partial^0Q_{r'}^+,
    \end{cases}
\end{equation*}
where $\tilde{F}$ is defined in \eqref{eq:TildeFDef}, with $f$ and $F$ replaced by $\bar{f}$ and $\bar{F}$ respectively. Furthermore, $\Tilde{F}  \in C^{k,\alpha}_p(Q_{r'}^+)$, so, by the inductive assumption (noticing that in the case $k=0$ we use Theorem \ref{teo-C^1,alpha}), by \eqref{eq:F:tilde:estimate}, \eqref{eq:C^3:i} and \eqref{eq:bar:f:est}, we obtain that $w\in C^{k+1,\alpha}_p(Q_r^+)$ and 
$$\|w\|_{C^{k+1,\alpha}_p(Q_r^+)}\le C \big(
\|u\|_{L^2(Q_1^+,y^a)}+
    \|f\|_{C^{k+1,\alpha}_p(Q_1^+)}+
    \|F\|_{C^{k+2,\alpha}_p(Q_1^+)}    
    \big),$$
 for some $C>0$ which depends only on     $N$, $a$, $\alpha$, $\lambda$, $\Lambda$, $\|A\|_{C^{k+2,\alpha}_p(Q_1^+)}$. Now, by the same arguments of Lemma \ref{lemma:approx:c2} and by Lemma \ref{lemma-2.4-TTV}, it follows
$$
(\partial_y u+ g/{\mu})(x,y,t)=\frac{1}{y^{a}}\int_{0}^y s^a w(x,s,t)ds,
$$ 
satisfies $\partial_{y}( \partial_y u+\bar{f}/\mu)\in C^{k+1,\alpha}$ and, by the regularity of $g$ and $\mu$, we deduce
$$[u_{yyy}]_{C^{k,\alpha}_p(Q_r^+)}+[u_{yy}]_{C^{k,\frac{1+\alpha}{2}}_t(Q_r^+)}\le C \big(
\|u\|_{L^2(Q_1^+,y^a)}+
    \|f\|_{C^{k+1,\alpha}_p(Q_1^+)}+
    \|F\|_{C^{k+2,\alpha}_p(Q_1^+)}    
    \big),$$
 for some $C>0$ depending only on $N$, $a$, $\alpha$, $\lambda$, $\Lambda$, $\|A\|_{C^{k+2,\alpha}_p(Q_1^+)}$.

To conclude the proof, it is sufficient to observe that
    $$ \partial_t u=y^{-a}{\div }(y^{a} (A\nabla u+F))+f\in C^{k+1,\alpha}_p(Q_r^+),$$
    which immediately implies 
    $$\|\partial_t u\|_{C^{k+1,\alpha}_p(Q_r^+)}\le C \big(
\|u\|_{L^2(Q_1^+,y^a)}+
    \|f\|_{C^{k+1,\alpha}_p(Q_1^+)}+
    \|F\|_{C^{k+2,\alpha}_p(Q_1^+)}    
    \big),$$
 for some $C>0$ depending only on     $N$, $a$, $\alpha$, $\lambda$, $\Lambda$, $\|A\|_{C^{k+2,\alpha}_p(Q_1^+)}$.
\end{proof}
\begin{proof}[Proof of Theorem \ref{teo1}.] Once established Proposition \ref{teo-C^k,alpha} and Lemma \ref{lemma:infinity-regularity}, our statement follows by approximation as in in Theorem \ref{teo1-C2}.
\end{proof}

\section{Cilindrically curved characteristic manifolds}\label{sec:curve}

In this section, we show how to extend the $C^{k+2,\alpha}_p$ regularity estimates to weak solutions of a class of equations having weights vanishing or exploding on curved characteristic manifolds $\Gamma$, as in \eqref{eq:1:curve}.

\begin{proof}[Proof of Corollary \ref{cor:C^1,alpha}]
The proof follows the one of \cite[Corollary 1.3]{AudFioVit24}: after composing with a standard local diffeomorphism one may apply the main Theorem \ref{teo1}.

Indeed, let us consider the classical diffeomorphism
      \[
      \Phi(x,y)=(x,y+\varphi(x)),
      \]
      which is of class $C^{k+2,\alpha}$ and then $C^{k+2,\alpha}_p$ extending constantly in the time variable.
      Up to a dilation, one has that $\tilde{u}:=u\circ(\Phi(x),t)$ is a weak solution to
\[
\begin{cases}
{\tilde\delta^a\partial_t\tilde{u}}-\div(\tilde\delta^a\tilde{A}\nabla \tilde{u})=\tilde\delta^a\tilde{f}+\div(\tilde\delta^a\tilde{F}),&\text{in }Q_1^+,\\
\displaystyle{\lim_{y\to0^+}}\tilde{\delta}^a(\tilde{A}\nabla\tilde{u}+\tilde{F})\cdot e_{N+1}=0 & \text{on }\partial^0Q_1^+.
\end{cases}   
\]   
      where $\tilde{\delta}=\delta\circ\Phi$, $\tilde{f}=f\circ(\Phi(x),t)$ and $\tilde{F}= J_\Phi^{-1} F\circ(\Phi(x),t)$ and  $\tilde{A}=(J_\Phi^{-1})(A\circ(\Phi(x),t))(J_\Phi^{-1})^T$.     
      We have that $\tilde{\delta} \in C^{k+2,\alpha}(B_1^+)$, $\tilde{A},\tilde{F} \in C^{k+1,\alpha}(B_1^+)$ and $\tilde{f} \in C^{k,\alpha}(B_1^+)$. Moreover, by using Lemma \ref{lemma-2.3-TTV}, $\tilde{\delta}$ satisfies
      $$\tilde{\delta}>0 \text{ in }B_1^+, \qquad \tilde{\delta}=0 \text{ on }\partial^0 B_1^+,\qquad \partial_y\tilde{\delta}>0\text{ on }\partial^0 B_1^+, \qquad\frac{\tilde{\delta}}{y}\in C^{k+1,\alpha}(B_1^+),\qquad \frac{\tilde{\delta}}{y}\ge\mu>0 \text{ in } \overline{B_1^+},$$
      where the last nondegeneracy condition is a consequence of the assumption $|\nabla\delta|\geq c_0>0$.
      
Defining $b(z):=(\tilde{\delta}(z)/y)^a\in C^{k+1,\alpha}(B_1^+)$, one has
\begin{align*}
0=&\int_{Q_1^+}y^ab\big(-\tilde{u}\partial_t \phi+\tilde{A}\nabla \tilde{u}\cdot\nabla \phi-\tilde{f}\phi+\tilde{F}\cdot\nabla\phi\big)\\
&=\int_{Q_1^+}y^a\big(
-\tilde{u}\partial_t(\phi b)+\tilde{A}\nabla \tilde{u}\cdot\nabla(\phi b)-\tilde{A}\nabla\tilde{u}\cdot\nabla b \phi-\tilde{f}(\phi b)+
\tilde{F}\cdot\nabla(\phi b) - \tilde{F}\cdot\nabla b \phi
\big),
\end{align*}
so, being $b\phi$ an admissible test function, we deduce that $\tilde{u}$ is a weak solution to
\[
\begin{cases}
{y^a\partial_t\tilde{u}}-\div(y^a\tilde{A}\nabla \tilde{u})=y^a\tilde{g}+\div(y^a\tilde{F}),&\text{in }Q_1^+,\\
\displaystyle{\lim_{y\to0^+}}y^a(\tilde{A}\nabla\tilde{u}+\tilde{F})\cdot e_{N+1}=0 & \text{on }\partial^0Q_1^+,
\end{cases}   
\]   
where 
\[
\tilde{g}:=\tilde{f}+\frac{\tilde{A}\D \tilde{u}\cdot \D b}{b}+\frac{\tilde{F}\cdot\D b}{b}.
\]
Finally, we apply a recursive argument to prove the $C^{k+2,\alpha}_p$-regularity of $\tilde u$, which in turns extends to the same regularity for the original $u$ by composing back with the diffeomorphism.

Let $k=0$. We notice that $u\in C^{1,\alpha}_p$ by \cite[Corollary 1.3]{AudFioVit24} and hence, after composing with the $C^{2,\alpha}_p$ diffeomorphism, one has $\nabla \tilde u\in C^{0,\alpha}_p$ which gives that $\tilde g\in C^{0,\alpha}_p$. Then, the $C^{2,\alpha}_p$-regularity of $\tilde u$ follows by Theorem \ref{teo1}. 

Finally, one may iterate this reasoning for any $k\geq1$ by replacing the use of the starting result \cite[Corollary 1.3]{AudFioVit24} with the present Corollary \ref{cor:C^1,alpha} at a lower step.
\end{proof}

\section{Parabolic higher order boundary Harnack principle}\label{sec:BH}
This last section, is devoted to the proof of the higher order boundary Harnack principle in Theorem \ref{BHk+2alpha}.

\begin{proof}[Proof of Theorem \ref{BHk+2alpha}]
First, the regularity assumptions of boundaries, coefficients and data for the equations in \eqref{BH} do guarantee that $u,v\in C^{k+2,\alpha}_\loc(\overline\Omega\cap Q_1)$, by classical theory of uniformly parabolic equations (for instance, see \cite{Lie96}). Hence, the equations in \eqref{BH} are satisfied both in the weak sense and pointwisely in $\Omega\cap Q_1$. From this, we deduce a pointwise equation for the quotient $w=v/u$ in $\Omega\cap Q_1$; that is,
\begin{equation}\label{eq:w:point}
u^2\partial_tw-\mathrm{div}(u^2A\nabla w)=uf-vg+u^2b\cdot\nabla w.
\end{equation}
Now, let us define the standard diffeomorphism
\[
      \Phi(x,y,t):=(x,y+\varphi(x,t),t),
 \]
      which is of class $C^{k+2,\alpha}_p$. Let us compose $u, v, f, g$ with $\Phi$; that is, $\tilde u=u\circ\Phi, \tilde v=v\circ\Phi, \tilde f=f\circ\Phi, \tilde g=g\circ\Phi$ and define
      $$\tilde A=(J_{z,\Phi}^{-1})^T(A\circ\Phi)J_{z,\Phi}^{-1},\qquad \tilde b=J_{z,\Phi}^{-1} b\circ\Phi,$$
      where $J_{z,\Phi}$ represents the square block $[c_{ij}]_{i,j=1,...,N+1}$ of the Jacobian $J_{\Phi}:=[c_{ij}]_{i,j=1,...,N+2}$.

Since $w$ solves \eqref{eq:w:point}, then, up to dilations, $\tilde w=w\circ\Phi=\tilde v/\tilde u$ solves
\begin{equation}\label{flatratio1}
y^2\mu^2\partial_t\tilde w-\mathrm{div}(y^2\mu^2\tilde A\nabla\tilde w)=y\left(\mu\tilde f-\frac{\tilde v}{y}\tilde g\right)+y^2\mu^2\tilde b\cdot\nabla\tilde w + y^2\mu^2c\cdot\nabla\tilde w,
\end{equation}
pointwisely in $Q_1^+$, where $\mu=\tilde u/y$ and $c=\partial_t\varphi e_{N+1}$. Now we need to do some remarks on regularity of the data of the weighted equation above. First, by Lemma \ref{lemma-2.3-TTV} and the non degeneracy condition $u(z,t)\geq c_0 \, d_p((z,t),\partial\Omega\cap Q_1)$ in \eqref{BH}, we can infer that
$$0<\overline c_0\leq\mu\in C^{k+1,\alpha}_p(B_1^+).$$
Thanks to the previous information, we can rewrite \eqref{flatratio1} dividing by $\mu^2$ as
\begin{equation}\label{flatratio2}
y^2\partial_t\tilde w-\mathrm{div}(y^2\tilde A\nabla\tilde w)=yh+y^2\overline b\cdot\nabla\tilde w,
\end{equation}
where
$$\overline b=\tilde b+c+2 \tilde A^T\frac{\nabla\mu}{\mu}\in C^{k,\alpha}_p,\qquad h=\frac{\mu\tilde f-\frac{\tilde v}{y}\tilde g}{\mu^2}\in C^{k+1,\alpha}_p.$$
Moreover, since
$$\tilde w=\frac{\tilde v/y}{\tilde u/y},$$
again by Lemma \ref{lemma-2.3-TTV} and the $C^{k+2,\alpha}_p$-regularity of $\tilde u,\tilde v$, we have $\tilde w\in C^{k+1,\alpha}_p(Q_1^+)$ which has two implications: first, the drift term in \eqref{flatratio2} can be considered as a forcing term; that is, $\overline b\cdot\nabla\tilde w=\overline f\in C^{k,\alpha}_p(Q_1^+)$;  secondly, $\tilde w$ belongs to $L^2(I_1;H^1(B_1^+,y^a)) \cap L^\infty(I_1;L^2(B_1^+,y^a))$ and, by multiplying the equation \eqref{flatratio2} by test functions $\phi\in C_c^\infty(Q_1^+)$ and integrating by parts, one gets that $\tilde{w}$ is a weak solution to
\begin{equation*}\label{flatratio3}
\begin{cases}
y^2\partial_t\tilde w-\mathrm{div}(y^2\tilde A\nabla\tilde w)=\mathrm{div}(y^2H)+y^2\overline f,&\text{in }Q_1^+,\\
\displaystyle{\lim_{y\to0^+}}y^2(\tilde{A}\nabla \tilde{w}+H)\cdot e_{N+1}=0, & \text{on }\partial^0Q_1^+,
\end{cases}
\end{equation*}
where the field
$$H(x,y,t)=\frac{e_{N+1}}{y^2}\int_0^ysh(x,s,t) \,ds$$
belongs to $C^{k+1,\alpha}_p(Q_1^+)$ by Lemma \ref{lemma-2.4-TTV}.

Then, the regularity $C^{k+2,\alpha}_p$-regularity of $\tilde w$ follows by Theorem \ref{teo1}. Finally, the same regularity is inherited by $w$ by composing back with the diffeomorphism.
\end{proof}

\section*{Acknowledgement}
The authors are research fellow of Istituto Nazionale di Alta Matematica INDAM group GNAMPA. A.A and G.F. are supported by the GNAMPA-INDAM project \emph{Teoria della regolarit\`a per problemi ellittici e parabolici con diffusione anisotropa e pesata}, CUP\_E53C22001930001. S.V. is supported by the MUR funding for Young Researchers - Seal of Excellence SOE\_0000194 \emph{(ADE) Anomalous diffusion equations: regularity and geometric properties of solutions and free boundaries}, and supported also by the PRIN project 2022R537CS \emph{$NO^3$ - Nodal Optimization, NOnlinear elliptic equations, NOnlocal geometric problems, with a focus on regularity}  by the GNAMPA-INDAM project \emph{Regolarit\`a e singolarit\`a in problemi di frontiere libere}, CUP\_E53C22001930001.



%
%

%
%
%

\end{document}